\documentclass[12pt]{amsart}
\usepackage{amssymb}
\usepackage{amsmath,amssymb,amsfonts,amsthm,graphics,
latexsym, amscd, amsfonts, epsfig, eepic,epic,psfrag,setspace}
\usepackage{mathrsfs}
\usepackage{color}
\usepackage{eucal}%    caligraphic-euler fonts: \mathcal{ }
\usepackage{eufrak}%   frak-euler        fonts: \mathfrak{ }
\usepackage[all]{xypic}
\usepackage{xspace}

\theoremstyle{plain}
\newtheorem{theorem}{Theorem}

\newtheorem{lemma}{Lemma}
\newtheorem{proposition}{Proposition}

\theoremstyle{definition}
\newtheorem{definition}{Definition}

\theoremstyle{example}

\newtheorem{remark}{Remark}
\theoremstyle{remark}

\numberwithin{equation}{section}

\begin{document}
%\doublespacing
%%%
%%%
%%%%%%%%%%%%%%%%%%%%%%%%%%%%%%%%%%%%%%%%%%%%%%%%%%%%%%%%%%%%%%%%%%%%%%%%%%
%%
%%%
\begin{center}
{\bf\Large Modular $k$-noncrossing diagrams}
\\
\vspace{15pt} Christian M. Reidys$^{\,\star}$, Rita R. Wang and 
Albus Y. Y. Zhao
\end{center}

\begin{center}
Center for Combinatorics, LPMC-TJKLC\\
Nankai University  \\
Tianjin 300071\\
         P.R.~China\\
         Phone: *86-22-2350-6800\\
         Fax:   *86-22-2350-9272\\
duck@santafe.edu
\end{center}

\centerline{\bf Abstract}{\small
In this paper we compute the generating function of modular,
$k$-noncrossing diagrams. A $k$-noncrossing diagram is called
modular if it does not contains any isolated arcs and any arc
has length at least four. Modular diagrams represent the 
deformation retracts of RNA pseudoknot structures 
\cite{Stadler:99,Reidys:07pseu,Reidys:07lego} and their 
properties reflect basic features of these bio-molecules.
The particular case of modular noncrossing diagrams has been 
extensively studied 
\cite{Waterman:78b,Waterman:79,Waterman:93,Schuster:98}. 
Let ${\sf Q}_k(n)$ denote the number of modular $k$-noncrossing 
diagrams over $n$ vertices. We derive exact enumeration results 
as well as the asymptotic formula ${\sf Q}_k(n)\sim c_k
n^{-(k-1)^2-\frac{k-1}{2}}\gamma_{k}^{-n}$ for $k=3, \ldots, 
9$ and derive a new proof of the formula ${\sf Q}_2(n)\sim 
1.4848\, n^{-3/2}\,1.8489^{-n}$ \cite{Schuster:98}.
}


%%%
%%%%%%%%%%%%%%%%%%%%%%%%%%%%%%%%%%%%%%%%%%%%%%%%%%%%%%%%%%%%%%%%%%%%%%%%%%%
%%%
\section{Introduction}\label{S:Intro}
%%%
%%%%%%%%%%%%%%%%%%%%%%%%%%%%%%%%%%%%%%%%%%%%%%%%%%%%%%%%%%%%%%%%%%%%%%%%%%%
%%%

The main result of this paper is the computation of the generating function 
of modular $k$-noncrossing diagrams, ${\bf Q}_k(z)$. 
A diagram is a labeled graph over the vertex set $[n]=\{1, \dots, n\}$ with 
vertex degrees not greater than one. The standard representation of a diagram
is derived by drawing its vertices in a horizontal line and its arcs 
$(i,j)$ in the upper half-plane.
A $k$-crossing is a set of $k$ distinct arcs
$(i_1, j_1), (i_2, j_2),\ldots,(i_k, j_k)$ with the property
\begin{equation*}
i_1 < i_2 < \ldots < i_k  < j_1 < j_2 < \ldots< j_k.
\end{equation*}
Similarly a $k$-nesting is a set $k$ distinct such arcs such that
\begin{equation*}
i_1 < i_2 < \ldots < i_k  < j_k < \ldots j_2 < j_1.
\end{equation*}
Let $A,B$ be two sets of arcs, then $A$ is nested in $B$ if any element
of $A$ is nested in any element of $B$.
A diagram without any $k$-crossings is called a $k$-noncrossing diagram. 
The length of an arc, $(i,j)$, is $s=j-i$, and we refer to such 
arc as $s$-arc. Furthermore,
\begin{itemize}
\item a stack of length $\sigma$, $S_{i,j}^{\sigma}$, is a maximal sequence
      of ``parallel'' arcs. A diagram, in which any arc is contained in a
      $s$-stack, where $s\ge 2$ is called a canonical diagram,
\begin{equation*}
((i,j),(i+1,j-1),\dots,(i+(\sigma-1),j-(\sigma-1))).
\end{equation*}
$S_{i,j}^{\sigma}$ is also referred to as a {$\sigma$-stack}.
\item a stem\index{stem} of size $s$ is a sequence
\begin{equation*}
\left(S_{i_1,j_1}^{\sigma_1},S_{i_2,j_2}^{\sigma_2},
\ldots,S_{i_{s},j_{s}}^{\sigma_{s}}\right)
\end{equation*}
where $S_{i_{m},j_{m}}^{\sigma_m}$ is nested in
$S_{i_{m-1},j_{m-1}}^{\sigma_{m-1}}$ such that any arc nested in
$S_{i_{m-1},j_{m-1}}^{\sigma_{m-1}}$ is either contained or nested
in $S_{i_{m},j_{m}}^{\sigma_m}$, for $2\leq m\leq s$.
\item a hairpin-loop (loop) consists of an arc $(i,j)$ and a sequence of
      consecutive isolated vertices $\{i+1,i+2,\ldots,j-1\}$, 
      see Fig.~\ref{F:diagram}.
\end{itemize}
%%%
%%%%%%%%%%%%%%%%%%%%%%%%%%%%%%%%%%%%%%%%%%%%%%%%%%%%%%%%%%%%%%%%%%%%%%
%%%
\begin{figure}[h!t!b!p]
\centering
\scalebox{0.9}[0.9]{\includegraphics{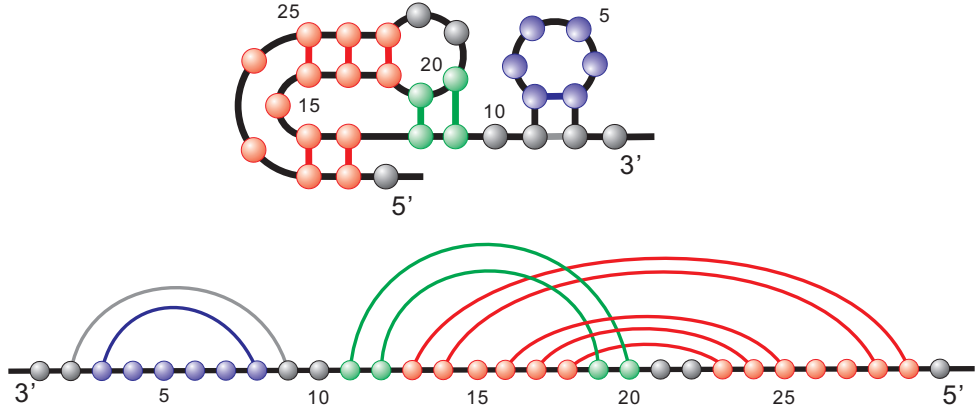}}
\caption{\small Features of a modular $3$-noncrossing diagram represented
as planar graph (top) and in standard representation (bottom).
We display a stack of length two (green), a stem of size two (red) and a 
hairpin-loop (purple).}
\label{F:diagram}
\end{figure}
%%%
%%%%%%%%%%%%%%%%%%%%%%%%%%%%%%%%%%%%%%%%%%%%%%%%%%%%%%%%%%%%%%%%%%%%%%
%%%

RNA secondary structures 
\cite{Waterman:78b,Waterman:79,Waterman:80,Waterman:94a} are in the 
language of diagrams exactly modular, $2$-noncrossing diagrams.
In \cite{Reidys:07pseu,Reidys:07asym,Reidys:07lego,Reidys:08ma}, various
classes of $k$-noncrossing diagrams have been enumerated. 
However the approach employed in these papers does not work for modular
$k$-noncrossing diagrams. In contrast to the situation for RNA secondary 
structures, the combination of minimum arc length and canonicity poses 
serious difficulties. 
The main idea is to build modular $k$-noncrossing diagrams via inflating 
their colored, ${\sf V}_k$-shapes, see Fig.~\ref{F:inflating}.
These shapes will be discussed in detail in Section~\ref{S:color}.
The inflation gives rise to ``stem-modules'' over shape-arcs and is the
key for the symbolic derivation of ${\bf Q}_k(z)$.
One additional observation maybe worth to be pointed out: 
the computation of the 
generating function of colored shapes in Section~\ref{S:color}, hinges on
the intuition that the crossings of short arcs are relatively simple and 
give rise to manageable recursions. The coloring of these shapes then allows 
to identify the arc-configurations that require special attention during 
the inflation process. Our results are of importance in the context of 
RNA pseudoknot structures \cite{Rietveld:82} and evolutionary optimization
\cite{Reidys-Stadler-Combinatorial landscapes-Siam Reviews}.
Furthermore they allow for conceptual proofs of the results in
\cite{Reidys:08han,Reidys:07asym,Reidys:07lego,Reidys:08ma}.

%%%
%%%%%%%%%%%%%%%%%%%%%%%%%%%%%%%%%%%%%%%%%%%%%%%%%%%%%%%%%%%%%%%%%%%%%%
%%%
\begin{figure}[h!t!b!p]
\centering
\scalebox{1.2}[1.2]{\includegraphics{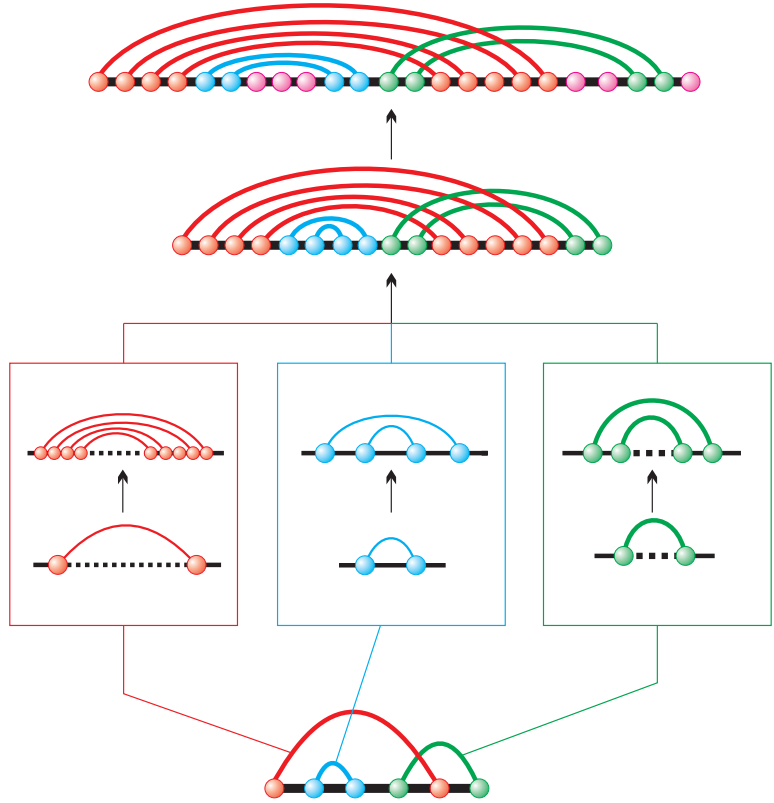}}
\caption{\small Modular $k$-noncrossing diagrams: the inflation method. 
A modular $3$-noncrossing diagram (top) is derived by inflating its 
${\sf V}_3$-shape (bottom) in two steps. 
First we individually inflate each shape-arc into a more complex 
configuration and second insert isolated vertices (purple).}
\label{F:inflating}
\end{figure}
%%%
%%%%%%%%%%%%%%%%%%%%%%%%%%%%%%%%%%%%%%%%%%%%%%%%%%%%%%%%%%%%%%%%%%%%%%
%%%

The paper is organized as follows. In Section~\ref{S:basics} we recall
some basic facts on singularity analysis, the generating function of
$k$-noncrossing matchings, ${\sf V}_k$-shapes and symbolic enumeration. 
Since the
results on $k$-noncrossing matchings are difficult to recover from the 
literature or simply new but immaterial for our purposes, we present 
their proofs in the Supplemental Materials (SM). In Section~\ref{S:k=2} 
we analyze modular, noncrossing diagrams and in 
Section~\ref{S:color} we compute the generating function of colored 
shapes. Finally we prove the main theorem in Section~\ref{S:main}.

%%%%
%%%%%%%%%%%%%%%%%%%%%%%%%%%%%%%%%%%%%%%%%%%%%%%%%%%%%%%%%%%%%%%%%
%%%%
\section{Some basic facts}\label{S:basics}
%%%%
%%%%%%%%%%%%%%%%%%%%%%%%%%%%%%%%%%%%%%%%%%%%%%%%%%%%%%%%%%%%%%%%%
%%%%
%%%%
%%%%%%%%%%%%%%%%%%%%%%%%%%%%%%%%%%%%%%%%%%%%%%%%%%%%%%%%%%%%%%%%%
%%%%
\subsection{Singularity analysis}
%%%%
%%%%%%%%%%%%%%%%%%%%%%%%%%%%%%%%%%%%%%%%%%%%%%%%%%%%%%%%%%%%%%%%%
%%%%
Oftentimes, we are given a generating function without having an
explicit formula of its coefficients. Singularity analysis is a framework
that allows to analyze the asymptotics of these coefficients.
The key to the asymptotics of the coefficients are the singularities,
which raises the question on how to locate them.
In the particular case of power series $f(z)=\sum_{n\geq 0}a_n\, z^n$ with
nonnegative coefficients and a radius of convergence $R>0$, a theorem of
Pringsheim \cite{Flajolet:07a,Tichmarsh:39}, guarantess a positive real
dominant singularity at $z=R$. As we are dealing here with combinatorial
generating functions we always have this dominant singularity.
We shall prove that for all our generating functions it is the
unique dominant singularity. The class of theorems that deal with the
deduction of information about coefficients from the generating function
are called transfer-theorems \cite{Flajolet:07a}.

%%%
%%%%%%%%%%%%%%%%%%%%%%%%%%%%%%%%%%%%%%%%%%%%%%%%%%%%%%%%%%%%
%%%
\begin{theorem}\label{T:transfer1}\cite{Flajolet:07a}
Let $[z^n]f(z)$ denote the $n$-th coefficient of the power series
$f(z)$ at $z=0$.\\
(a) Suppose $f(z)=(1-z)^{-\alpha}$, $\alpha\in\mathbb{C}\setminus
\mathbb{Z}_{\le 0}$, then
\begin{eqnarray}
[z^n]\, f(z) & \sim & \frac{n^{\alpha-1}}{\Gamma(\alpha)}\left[
1+\frac{\alpha(\alpha-1)}{2n}+\frac{\alpha(\alpha-1)(\alpha-2)(3\alpha-1)}
{24 n^2}+\right.\nonumber \\
&&\qquad  \quad
\left.\frac{\alpha^2(\alpha-1)^2(\alpha-2)(\alpha-3)}{48n^3}+
O\left(\frac{1}{n^4}\right)\right].
\end{eqnarray}
(b) Suppose $f(z)=(1-z)^{r}\log(\frac{1}{1-z})$,
$r\in\mathbb{Z}_{\ge 0}$, then we have
\begin{equation}
[z^n]f(z)\sim \,  (-1)^r\frac{r!}{n(n-1)\dots(n-r)}.
\end{equation}
\end{theorem}
%%%
%%%%%%%%%%%%%%%%%%%%%%%%%%%%%%%%%%%%%%%%%%%%%%%%%%%%%%%%%%%%
%%%
We use the notation
\begin{equation}\label{E:genau}
\left(f(z)=\Theta\left(g(z)\right) \
\text{\rm as $z\rightarrow \rho$}\right)\  \Longleftrightarrow \
\left(f(z)/g(z)\rightarrow c\ \text{\rm as $z\rightarrow \rho$}\right),
\end{equation}
where $c$ is some constant. We say a function $f(z)$ is
$\Delta_\rho$ analytic at its dominant singularity $z=\rho$, if it
analytic in some domain $\Delta_\rho(\phi,r)=\{ z\mid \vert z\vert <
r, z\neq r,\, \vert {\rm Arg}(z-\rho)\vert >\phi\}$, for some
$\phi,r$, where $r>|\rho|$ and $0<\phi<\frac{\pi}{2}$. Since the
Taylor coefficients have the property
\begin{equation}\label{E:scaling}
\forall \,\gamma\in\mathbb{C}\setminus 0;\quad [z^n]f(z)=\gamma^n
[z^n]f(\frac{z}{\gamma}),
\end{equation}
We can, without loss of generality, reduce our analysis to the case
where $z=1$ is the unique dominant singularity. The next theorem
transfers the asymptotic expansion of a function near its unique
dominant singularity to the asymptotic of the function's
coefficients.

%%%
%%%%%%%%%%%%%%%%%%%%%%%%%%%%%%%%%%%%%%%%%%%%%%%%%%%%%%%%%%%%%%%%%%%%%%%%%
%%%
%%%
%%%%%%%%%%%%%%%%%%%%%%%%%%%%%%%%%%%%%%%%%%%%%%%%%%%%%%%%%%%%%%%%%%%%%%%%%
%%%
\begin{theorem}\label{T:transfer1b}{\bf }\cite{Flajolet:07a}
Let $f(z)$ be a $\Delta_1$ analytic function at its unique dominant 
singularity $z=1$. Let 
$$g(z)=(1-z)^{\alpha}\log^{\beta}\left(\frac{1}{1-z}\right)
,\quad\alpha,\beta\in \mathbb{R}.
$$ 
That is we have in the intersection of a neighborhood of $1$
\begin{equation}\label{T:transfer2}
f(z) = \Theta(g(z)) \quad \text{\it for } z\rightarrow 1.
\end{equation}
Then we have
\begin{equation}
[z^n]f(z)= \Theta\left([z^n]g(z)\right).
\end{equation}
\end{theorem}
%%%
%%%%%%%%%%%%%%%%%%%%%%%%%%%%%%%%%%%%%%%%%%%%%%%%%%%%%%%%%%%%%%%%%%%%%%%%%
%%%

%%%
%%%%%%%%%%%%%%%%%%%%%%%%%%%%%%%%%%%%%%%%%%%%%%%%%%%%%%%%%%%%%%%%%%%%%%%%%
%%%
\subsection{$k$-noncrossing matchings}
%%%
%%%%%%%%%%%%%%%%%%%%%%%%%%%%%%%%%%%%%%%%%%%%%%%%%%%%%%%%%%%%%%%%%%%%%%%%%
%%%

A $k$-noncrossing matching is a $k$-noncrossing diagram without isolated
points. Let $f_k(2n)$ denote the number of $k$-noncrossing matchings.
The exponential generating function of $k$-noncrossing matchings
satisfies the following identity \cite{Chen,Grabiner:93a,Reidys:07pseu}
\begin{eqnarray}\label{E:ww0}
\label{E:ww1} \sum_{n\ge 0} f_{k}(2n)\cdot\frac{z^{2n}}{(2n)!} & =
& \det[I_{i-j}(2z)-I_{i+j}(2z)]|_{i,j=1}^{k-1}
\end{eqnarray}
where $I_{r}(2z)=\sum_{j \ge 0}\frac{z^{2j+r}}{{j!(j+r)!}}$ is the
hyperbolic Bessel function of the first kind of order $r$.
Eq.~(\ref{E:ww0}) combined with the fact that recursions for the coefficients
of the exponential generating function translate into recursions
for the coefficients of the ordinary generating function,
allows us to prove:

%%%
%%%%%%%%%%%%%%%%%%%%%%%%%%%%%%%%%%%%%%%%%%%%%%%%%%%%%%%%%%%%%%%%%%%%%%%%%
%%%
\begin{lemma}\label{L:k-D}
The generating function of $k$-noncrossing
matchings over $2n$ vertices, {${\bf
F}_k(z)=\sum_{n\ge 0}f_k(2n)\,z^{n}$} is $D$-finite, \cite{Stanley:80},
i.e.\ there exists some $e\in \mathbb{N}$
such that
\begin{equation}\label{E:KJ}
q_{0,k}(z)\frac{d^e}{d z^e}{\bf F}_k(z)+q_{1,k}(z)\frac{d^{e-1}}{d
z^{e-1}}{\bf F}_k(z)+\cdots+q_{e,k}(z){\bf F}_k(z)=0,
\end{equation}
where $q_{j,k}(z)$ are polynomials.
\end{lemma}
%%%
%%%%%%%%%%%%%%%%%%%%%%%%%%%%%%%%%%%%%%%%%%%%%%%%%%%%%%%%%%%%%%%%%%%%%%%%%
%%%
We sketch the proof of Lemma~\ref{L:k-D} in the SM.

Lemma~\ref{L:k-D} is of importance for two reasons: first any singularity of
${\bf F}_k(z)$ is contained in the set of roots of $q_{0,k}(z)$
\cite{Stanley:80}, which we denote by $R_k$.
Second, the specific form of the ODE in eq.~(\ref{E:KJ}) is the key to derive
the singular expansion of ${\bf F}_k(z)$, see Proposition~\ref{P:fk} below.

We proceed by computing for $2\leq k\leq 9$, the polynomials $q_{0,k}(z)$
and their roots, see Table \ref{Table:polyroot} and
%%%
%%%%%%%%%%%%%%%%%%%%%%%%%%%%%%%%%%%%%%%%%%%%%%%%%%%%%%%%%%%%%%%%%%%%%%%%%
%%%
\begin{table}
\begin{center}
\begin{tabular}{cll}
\hline
$k$ & $q_{0,k}(z)$ & $R_k$  \\
\hline
$2$ & $(4z-1)z$ & $\{\frac{1}{4}\}$\\
$3$ & $(16z-1)z^2$ & $\{\frac{1}{16}\}$\\
$4$ & $(144z^2-40z+1)z^3$ & $\{\frac{1}{4},\frac{1}{36}\}$\\
$5$ & $(1024z^2-80z+1)z^4$ & $\{\frac{1}{16},\frac{1}{64}\}$\\
$6$ & $(14400z^3-4144z^2+140z-1)z^5$ &
$\{\frac{1}{4},\frac{1}{36},\frac{1}{100}\}$\\
$7$ & $(147456z^3-12544z^2+224z-1)z^6$ &
$\{\frac{1}{16},\frac{1}{64},\frac{1}{144}\}$\\
$8$ &$(2822400z^4-826624z^3+31584z^2-336z+1)z^{7}$&
$\{\frac{1}{4},\frac{1}{36},\frac{1}{100},\frac{1}{196}\}$\\
$9$ &$(37748736z^4-3358720z^3+69888z^2-480z+1)z^{8}$&
$\{\frac{1}{16},\frac{1}{64},\frac{1}{144},\frac{1}{256}\}$\\
\hline
\end{tabular}
\centerline{}
\smallskip
\caption{\small We present the polynomials $q_{0,k}(z)$  and their
nonzero roots obtained by the MAPLE package {\tt GFUN}. }
\label{Table:polyroot}
\end{center}
\end{table}
%%%
%%%%%%%%%%%%%%%%%%%%%%%%%%%%%%%%%%%%%%%%%%%%%%%%%%%%%%%%%%%%%%%%%%%%%%%%%
%%%
observe that \cite{Reidys:08k}
\begin{equation}\label{E:theorem}
f_{k}(2n) \, \sim  \, \widetilde{c}_k  \, n^{-((k-1)^2+(k-1)/2)}\,
(2(k-1))^{2n},\qquad \widetilde{c}_k>0,\, k\ge 2.
\end{equation}
Equation~(\ref{E:theorem}) guarantees
that ${\bf F}_k(z)$ has the unique dominant singularity $\rho_k^2$, where
$\rho_k=1/(2k-2)$.
According to Lemma~\ref{L:k-D}, ${\bf F}_k(z)$ is $D$-finite, whence we have
analytic continuation in any simply connected domain containing zero avoiding
its singularities \cite{Wasow:87}. As a result ${\bf F}_k(z)$ is
$\Delta_{\rho_k^2}$ analytic as required by Theorem~\ref{T:transfer1b}.
Lemma~\ref{L:k-D} and eq.~(\ref{E:theorem}) put us in position to present
the singular expansion of ${\bf F}_k(z)$:
%%%
%%%%%%%%%%%%%%%%%%%%%%%%%%%%%%%%%%%%%%%%%%%%%%%%%%%%%%%%%%%%%%%%%%%%%%%%%
%%%
\begin{proposition}\label{P:fk}
For $2\le k\le 9$, the singular expansion of ${\bf F}_k(z)$
for $z\rightarrow \rho_k^2$ is given by
\begin{equation*}
\begin{split}
{\bf F}_k(z)=
\begin{cases}
P_k(z-\rho_k^2)+c_k'(z-\rho_k^2)^{((k-1)^2+(k-1)/2)-1}
\log^{}(z-\rho_k^2)\left(1+o(1)\right)  \\
P_k(z-\rho_k^2)+c_k'(z-\rho_k^2)^{((k-1)^2+(k-1)/2)-1}
\left(1+o(1)\right),
\end{cases}
\end{split}
\end{equation*}
depending on $k$ being odd or even. Furthermore, the terms $P_k(z)$ are
polynomials of degree not larger than $(k-1)^2+(k-1)/2-1$, $c_k'$ is
some constant, and $\rho_k=1/(2k-2)$.
\end{proposition}
%%%
%%%%%%%%%%%%%%%%%%%%%%%%%%%%%%%%%%%%%%%%%%%%%%%%%%%%%%%%%%%%%%%%%%%%%%%%%
%%%
We give the proof of Proposition~\ref{P:fk} is the SM.

In our derivations the following instance of the supercritical paradigm
\cite{Flajolet:07a} is of central importance: we
are given a $D$-finite function, $f(z)$ and an algebraic function
$g(u)$ satisfying $g(0)=0$. Furthermore we suppose that $f(g(u))$
has a unique real valued dominant singularity $\gamma$ and $g$ is
regular in a disc with radius slightly larger than $\gamma$. Then
the supercritical paradigm stipulates that the subexponential
factors of $f(g(u))$ at $u=0$, given that $g(u)$ satisfies certain
conditions, coincide with those of $f(z)$.

Theorem \ref{T:transfer1}, Theorem \ref{T:transfer1b} and
Proposition~\ref{P:fk} allow under certain conditions to obtain the
asymptotics of the coefficients of supercritical compositions of the
``outer'' function ${\bf{F}}_k(z)$ and ``inner'' function $\psi(z)$.

%%%
%%%%%%%%%%%%%%%%%%%%%%%%%%%%%%%%%%%%%%%%%%%%%%%%%%%%%%%%%%%%%%%%%%%%%%%%%
%%%
\begin{proposition}\label{P:algeasym}
Let $\psi(z)$ be an algebraic, analytic function in a
domain $\mathcal{D}=\{z||z|\leq r\}$ such that
$\psi(0)=0$. Suppose $\gamma$ is the unique
dominant singularity of ${\bf F}_k(\psi(z))$ and minimum positive real
solution of $\psi(\gamma)=\rho_k^2$, $|\gamma|<r$,
where $\psi'(\gamma)\neq 0$.
Then ${\bf F}_k(\psi(z))$ has a singular expansion and
\begin{equation}\label{E:LL}
[z^n]{\bf F}_k(\psi(z)) \sim A\,n^{-((k-1)^2+(k-1)/2)}
\left(\frac{1}{\gamma}\right)^n,
\end{equation}
where $A$ is some constant.
\end{proposition}
%%%
%%%%%%%%%%%%%%%%%%%%%%%%%%%%%%%%%%%%%%%%%%%%%%%%%%%%%%%%%%%%%%%%%%%%%%%%
%%%

%%%
%%%%%%%%%%%%%%%%%%%%%%%%%%%%%%%%%%%%%%%%%%%%%%%%%%%%%%%%%%%%%%%%%%%%%%%%
%%%
\subsection{Shapes}
%%%
%%%%%%%%%%%%%%%%%%%%%%%%%%%%%%%%%%%%%%%%%%%%%%%%%%%%%%%%%%%%%%%%%%%%%%%%
%%%

%%%
%%%%%%%%%%%%%%%%%%%%%%%%%%%%%%%%%%%%%%%%%%%%%%%%%%%%%%%%%%%%%%%%%%%%%%%%%%
%%%
\begin{definition}
A ${\sf V}_k$-shape is a $k$-noncrossing matching having stacks of length
exactly one.
\end{definition}
%%%
%%%%%%%%%%%%%%%%%%%%%%%%%%%%%%%%%%%%%%%%%%%%%%%%%%%%%%%%%%%%%%%%%%%%%%%%%%
%%%
In the following we refer to ${\sf V}_k$-shape simply as shapes.
That is, given a modular, $k$-noncrossing diagram, $\delta$, its shape 
is obtained by first replacing each stem by an arc and then removing all 
isolated vertices, see Fig.~\ref{F:Ikmapchain}.
%%%%%%%%%%%%%%%%%%%%%%%%%%%%%%%%%%%
%%%%%%%%%%%%%%%%%%%%%%%%%%%%%%%%%%%%%%%%%%%%%%%%%%%%%%%%%%%%%%%%%%%%%%
%%%%%%%%%%%%%%%%%%%%%%%%%%%%%%%%%%%
\begin{figure}[ht]
\centerline{\includegraphics[width=0.65\textwidth]{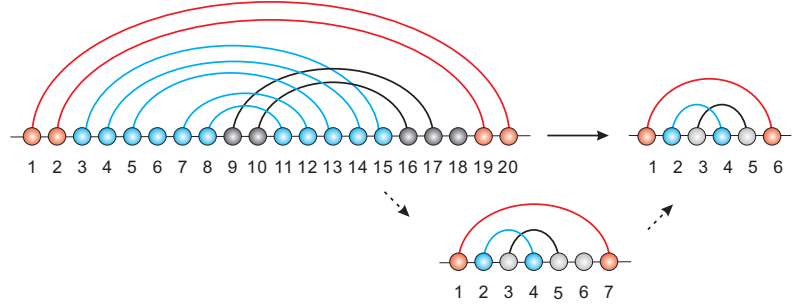}}
\caption{\small From diagrams to shapes: A
modular, $3$-noncrossingd diagram (top-left) is mapped in two steps
into its ${\sf V}_3$-shape (top-right). A stem (blue) is replaced by 
an single shape-arc (blue).}
\label{F:Ikmapchain}
\end{figure}
%%%%%%%%%%%%%%%%%%%%%%%%%%%%%%%%%%%
%%%%%%%%%%%%%%%%%%%%%%%%%%%%%%%%%%%%%%%%%%%%%%%%%%%%%%%%%%%%%%%%%%%%%%
%%%%%%%%%%%%%%%%%%%%%%%%%%%%%%%%%%%

Let ${\mathcal I}_k(s,m)$ ($i_k(s,m)$) denote the set (number) of
the ${\sf V}_k$-shapes with $s$ arcs and $m$ $1$-arcs having the bivariate
generating function
\begin{equation}
{\bf I}_k(z,u)=\sum_{s\geq0}\sum_{m=0}^{s} i_k(s,m)z^su^m.
\end{equation}
The bivariate generating function of $i_k(s,m)$ and the
generating function of ${\bf F}_k(z)$ are related as follows:
%%%
%%%%%%%%%%%%%%%%%%%%%%%%%%%%%%%%%%%%%%%%%%%%%%%%%%%%%%%%%%%%%%%%%%%%%%%%%%%%
%%%
\begin{lemma}\label{T:gfIk}\cite{Reidys:09shape}
Let $k$ be natural number where $k\geq 2$, then the generating
function ${\bf I}_k(z,u)$ satisfy
\begin{eqnarray}\label{E:gfIk}
{\bf I}_k(z,u) & = & \frac{1+z}{1+2z-zu}
{\bf F}_k\left(\frac{z(1+z)}{(1+2z-zu)^2}\right).
\end{eqnarray}
\end{lemma}

%%%
%%%%%%%%%%%%%%%%%%%%%%%%%%%%%%%%%%%%%%%%%%%%%%%%%%%%%%%%%%%%%%%%%%%%%%%%%%%%
%%%
\subsection{Symbolic enumeration}
%%%
%%%%%%%%%%%%%%%%%%%%%%%%%%%%%%%%%%%%%%%%%%%%%%%%%%%%%%%%%%%%%%%%%%%%%%%%%%%%
%%%
In the following we will compute the generating functions via the
symbolic enumeration method \cite{Flajolet:07a}. For this purpose we
need the notion of a combinatorial class. A combinatorial class
$(\mathcal{C},w_{\mathcal{C}})$ is a set together with a
size-function, $w_{\mathcal{C}}\colon \mathcal{C}\longrightarrow
\mathbb{Z}^+$ such that $\mathcal{C}_n=w_{\mathcal{C}}^{-1}(n)$ is
finite for any $n\in\mathbb{Z}^+$. We write $w$ instead of
$w_{\mathcal{C}}$ and set $C_n=\vert \mathcal{C}_n\vert$. Two
special combinatorial classes are $\mathcal{E}$ and $\mathcal{Z}$
which contain only one element of size $0$ and $1$, respectively.
The generating function of a combinatorial class $\mathcal{C}$ is
given by
\begin{equation}
{\bf{C}}(z)=\sum_{c\in\mathcal{C}}z^{w(c)}=\sum_{n\geq 0}C_n\, z^n,
\end{equation}
where $\mathcal{C}_n\subset \mathcal{C}$. In particular, the
generating functions of the classes $\mathcal{E}$ and $\mathcal{Z}$
are ${\bf{E}}(z)=1$ and ${\bf{Z}}(z)=z$. Suppose
$\mathcal{C},\mathcal{D}$ are combinatorial classes. Then
$\mathcal{C}$ is isomorphic to $\mathcal{D}$,
$\mathcal{C}\cong\mathcal{D}$, if and only if $\forall\, n\geq
0$,$|\mathcal{C}_n|=|\mathcal{D}_n|$. In the following we shall
identify isomorphic combinatorial classes and write
$\mathcal{C}=\mathcal{D}$ if $\mathcal{C}\cong\mathcal{D}$. We set
\begin{itemize}
\item $\mathcal{C}+\mathcal{D}:=\mathcal{C}\cup\mathcal{D}$, if
$\mathcal{C}\cap\mathcal{D}=\varnothing$ and for
$\alpha\in\mathcal{C}+\mathcal{D}$,
\begin{equation}
w_{\mathcal{C}+\mathcal{D}}(\alpha)=\left\{
\begin{aligned}
&w_{\mathcal{C}}(\alpha)\quad \text{if}\ \alpha\in\mathcal{C}\\
&w_{\mathcal{D}}(\alpha)\quad \text{if}\ \alpha\in\mathcal{D}.
\end{aligned} \right.
\end{equation}
\item $\mathcal{C}\times\mathcal{D}:=\{\alpha=(c,d)\mid
c\in\mathcal{C},d\in\mathcal{D}\}$ and for
$\alpha\in\mathcal{C}\times\mathcal{D}$,
\begin{equation}
w_{\mathcal{C}\times\mathcal{D}}((c,d))=
w_{\mathcal{C}}(c)+w_{\mathcal{D}}(d).
\end{equation}
\end{itemize}
and furthermore $\mathcal{C}^m:=\prod_{h=1}^m\mathcal{C}$ and
$\textsc{Seq}(\mathcal{C}):={\mathcal{E}}+\mathcal{C}+\mathcal{C}^2+\cdots$.
Plainly, $\textsc{Seq}(\mathcal{C})$ is a combinatorial class if and
only if there is no element in $\mathcal{C}$ of size $0$.
We immediately observe
%%%
%%%%%%%%%%%%%%%%%%%%%%%%%%%%%%%%%%%%%%%%%%%%%%%%%%%%%%%%%%%%%%%%%%%%
%%%
\begin{proposition}
Suppose $\mathcal{A}$, $\mathcal{C}$ and $\mathcal{D}$ are
combinatorial classes with generating functions ${\bf{A}(z)}$,
${\bf{C}}(z)$ and ${\bf{D}}(z)$. Then\\
{\sf (a)}  $\mathcal{A}=\mathcal{C}+\mathcal{D}\Longrightarrow
{\bf{A}}(z)={\bf{C}}(z)+{\bf{D}}(z)$\\
{\sf (b)} $\mathcal{A}=\mathcal{C}\times\mathcal{D}\Longrightarrow
{\bf{A}}(z)={\bf{C}}(z)\cdot {\bf{D}}(z)$\\
{\sf (c)} $\mathcal{A}=\textsc{Seq}(\mathcal{C})\Longrightarrow
{\bf{A}}(z)=\frac{1}{1-{\bf{C}}(z)}$.
\end{proposition}

%%%
%%%%%%%%%%%%%%%%%%%%%%%%%%%%%%%%%%%%%%%%%%%%%%%%%%%%%%%%%%%%%%%%%%%%%%%%
%%%

%%%
%%%
%%%%%%%%%%%%%%%%%%%%%%%%%%%%%%%%%%%%%%%%%%%%%%%%%%%%%%%%%%%%%%%%%%%%%%%%%
%%%

%%%
%%%%%%%%%%%%%%%%%%%%%%%%%%%%%%%%%%%%%%%%%%%%%%%%%%%%%%%%%%%%%%%%%%%%%%%%%%%%

%%%%%%%%%%%%%%%%%%%%%%%%%%%%%%%%%%%%%%%%%%%%%%%%%%%%%%%%%%%%%%%%%%%%%%%%

%%%
%%%%%%%%%%%%%%%%%%%%%%%%%%%%%%%%%%%%%%%%%%%%%%%%%%%%%%%%%%%%%%%%%%%%%%%%%
%%%

\section{Modular, noncrossing diagrams}\label{S:k=2}

%%%
%%%%%%%%%%%%%%%%%%%%%%%%%%%%%%%%%%%%%%%%%%%%%%%%%%%%%%%%%%%%%%%%%%%%%%%%%%
%%%

Let us begin by studying first the case $k=2$ \cite{Schuster:98},
where the asymptotic formula
\begin{equation*}
{\sf Q}_2(n)\sim 1.4848\cdot n^{-3/2}\cdot 1.8489^n,
\end{equation*}
has been derived. In the following we extend the result in \cite{Schuster:98} 
by computing the generating function explicitely. The above asymptotic formula 
follows then easily by means of singularity analysis.

%%%
%%%%%%%%%%%%%%%%%%%%%%%%%%%%%%%%%%%%%%%%%%%%%%%%%%%%%%%%%%%%%%%%%%%%%
%%%
\begin{proposition}\label{P:k=2}
The generating function of modular, noncrossing diagrams is given by
\begin{equation}\label{E:UU}
{\bf Q}_2(z)=\frac{1-z^2+z^4}{1 - z - z^2 + z^3 + 2 z^4 + z^6}\cdot {\bf F}_2
\left(\frac{z^4 - z^6 + z^8}{(1 - z - z^2 + z^3 + 2 z^4 + z^6)^2}\right)
\end{equation}
and the coefficients ${\sf Q}_2(n)$ satisfy
\begin{equation*}
{\sf Q}_2(n)\sim c_2 n^{-3/2} \gamma_{2}^{-n},
\end{equation*}
where $\gamma_{2}$ is the minimal, positive real solution
of $\vartheta(z)=1/4$, and
\begin{equation}
\vartheta(z)=
\frac{z^4 - z^6 + z^8}{(1 - z - z^2 + z^3 + 2 z^4 + z^6)^2}.
\end{equation}
Here we have $\gamma_{2}\approx 1.8489$ and
$c_2\approx 1.4848$.
\end{proposition}
%%%
%%%%%%%%%%%%%%%%%%%%%%%%%%%%%%%%%%%%%%%%%%%%%%%%%%%%%%%%%%%%%%%%%%%%%
%%%
\begin{proof}
Let ${\mathcal Q}_{2}$ denote the set of modular noncrossing diagrams,
${\mathcal I}_2$ the set of all ${\sf V}_2$-shapes and ${\mathcal I}_2(m)$ 
those having exactly $m$ $1$-arcs. Then we have the surjective map
\begin{equation*}
\varphi\colon {\mathcal Q}_{2}\rightarrow {\mathcal I}_2.
\end{equation*}
The map $\varphi$ is obviously surjective, inducing the partition 
${\mathcal Q}_{2}=\dot\cup_\gamma\varphi^{-1}(\gamma)$, where 
$\varphi^{-1}(\gamma)$ is the preimage set of shape $\gamma$ under 
the map $\varphi$. Accordingly, we arrive at
\begin{equation}\label{E:Hgf}
{\bf Q}_{2}(z) =
\sum_{m\geq 0}\sum_{\gamma\in\,{\mathcal I}_2(m)}
\mathbf{Q}_{\gamma}(z).
\end{equation}
We proceed by computing the generating function $\mathbf{Q}_{\gamma}(z)$.
We shall construct $\mathbf{Q}_{\gamma}(z)$ from certain combinatorial classes
as ``building blocks''. The latter are: $\mathcal{M}$ (stems),
$\mathcal{K}$ (stacks), $\mathcal{N}$ (induced stacks),
$\mathcal{L}$ (isolated vertices), $\mathcal{R}$ (arcs) and $\mathcal{Z}$
(vertices), where $\mathbf{Z}(z)=z$ and $\mathbf{R}(z)=z^2$.
We inflate $\gamma\in {\mathcal I}_2(m)$ having $s$ arcs, where
$s\geq \max\{1,m\}$, to a modular noncrossing diagram in two steps.\\
{\bf Step I:} we inflate any shape-arc to a stack of length at least
$2$ and subsequently add additional stacks. The latter are called
induced stacks \index{induced stack} and have to be separated by means of
inserting isolated vertices, see Fig.~\ref{F:addstack}.
%%%%%%%%%%%%%%%%%%%%%%%%%%%%%%%%%%%%%%%%%%%%%%%%%%%%%%%%%%%%%%%%%%%%%%%%
%%%%%%%%%%%%%%%%%%%%%%%%%%%%%%%%%%%%%%%%%%%%%%%%%%%%%%%%%%%%%%%%%%%%%%%%
\begin{figure}[ht]
\centerline{\includegraphics[width=0.9\textwidth]{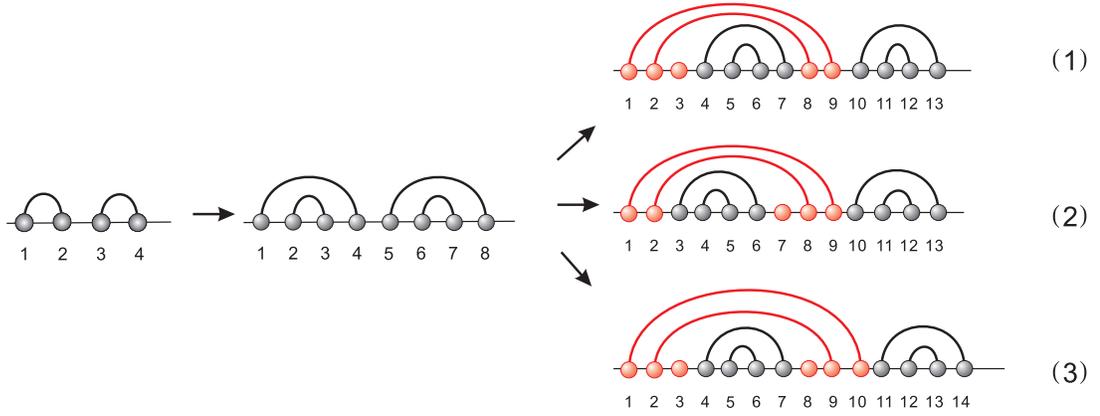}}
\caption{\small A shape (left) is inflated to a noncrossing, canonical diagram.}
\label{F:addstack}
\end{figure}
%%%%%%%%%%%%%%%%%%%%%%%%%%%%%%%%%%%%%%%%%%%%%%%%%%%%%%%%%%%%%%%%%%%%%%%%
%%%%%%%%%%%%%%%%%%%%%%%%%%%%%%%%%%%%%%%%%%%%%%%%%%%%%%%%%%%%%%%%%%%%%%%%
Note that during this first inflation step no intervals of isolated
vertices, other than those necessary for separating the nested stacks
are inserted. We generate
\begin{itemize}
\item sequences of isolated vertices 
$\mathcal{L}= \textsc{Seq}(\mathcal{Z})$, where
\begin{eqnarray*}
 {\bf L}(z) & = &  \frac{1}{1-z}
\end{eqnarray*}
\item stacks, i.e.
\begin{equation*}
\mathcal{K}=
\mathcal{R}^{2}\times\textsc{Seq}\left(\mathcal{R}\right)
\end{equation*}
with the generating function
\begin{eqnarray*}
\mathbf{K}(z) & = & z^{4}\cdot \frac{1}{1-z^2},
\end{eqnarray*}
\item induced stacks, i.e.~stacks together with at least one nonempty
interval of isolated vertices on either or both its sides.
\begin{equation*}
\mathcal{N}=\mathcal{K}\times
\left(\mathcal{Z}\times\mathcal{L}
+\mathcal{Z}\times\mathcal{L}+\left(\mathcal{Z}\times
\mathcal{L}\right)^2\right)
\end{equation*}
with generating function
\begin{equation*}
\mathbf{N}(z)=\frac{z^{4}}{1-z^2}\left(2\frac{z}{1-z}
+\left(\frac{z}{1-z}\right)^2\right),
\end{equation*}
\item stems\index{stem}, that is pairs consisting of a stack $\mathcal{K}$
and an arbitrarily long sequence of induced stacks
\begin{equation*}
\mathcal{M}=\mathcal{K}
\times \textsc{Seq}\left(\mathcal{N}\right)
\end{equation*}
with generating function
\begin{eqnarray*}
\mathbf{M}(z)=\frac{\mathbf{K}(z)}{1-\mathbf{N}(z)}=
\frac{\frac{z^{4}}{1-z^2}}
{1-\frac{z^{4}}{1-z^2}\left(2\frac{z}{1-z}
+\left(\frac{z^3}{1-z}\right)^2\right)}.
\end{eqnarray*}
\end{itemize}
{\bf Step II:} we insert additional isolated vertices at the
remaining $(2s+1)$ positions. For each $1$-arc at least three such
isolated vertices are necessarily inserted, see
Fig.~\ref{F:addvertex}.
%%%%%%%%%%%%%%%%%%%%%%%%%%%%%%%%%%%%%%%%%%%%%%%%%%%%%%%%%%%%%%%%%%%%%%%%
%%%%%%%%%%%%%%%%%%%%%%%%%%%%%%%%%%%%%%%%%%%%%%%%%%%%%%%%%%%%%%%%%%%%%%%%
\begin{figure}[ht]
\centerline{\includegraphics[width=1\textwidth]{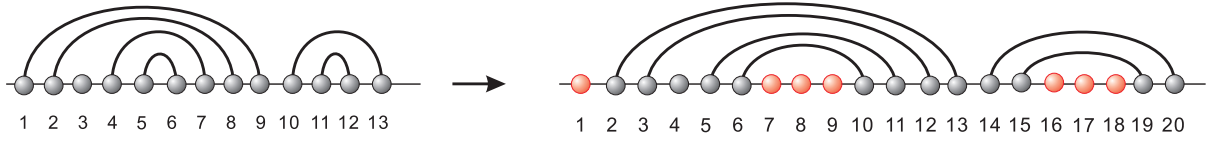}}
\caption{\small Step II: the noncrossing canonical diagram (left) 
obtained in {\sf (1)} in Fig.~\ref{F:addstack} is inflated to a 
modular noncrossing diagram (right) by adding isolated vertices (red).}
\label{F:addvertex}
\end{figure}
%%%%%%%%%%%%%%%%%%%%%%%%%%%%%%%%%%%%%%%%%%%%%%%%%%%%%%%%%%%%%%%%%%%%%%%%
%%%%%%%%%%%%%%%%%%%%%%%%%%%%%%%%%%%%%%%%%%%%%%%%%%%%%%%%%%%%%%%%%%%%%%%%
Combining Step I and Step II we arrive at
\begin{equation}
\mathcal{Q}_{\gamma}=\left(\mathcal{M}\right)^s
\times\mathcal{L}^{2s+1-m}\times\left(\mathcal{Z}^3\times
\mathcal{L}\right)^{m},
\end{equation}
where $\mathcal{Q}_{\gamma}$ is the combinatorial class of modular
noncrossing diagrams having shape $\gamma$. Therefore,
\begin{eqnarray*}
{\bf Q}_{\gamma}(z)&=&
\left(\frac{\frac{z^{4}}{1-z^2}}
{1-\frac{z^{4}}{1-z^2}\left(2\frac{z}{1-z}
+\left(\frac{z}{1-z}\right)^2\right)}\right)^s
\left(\frac{1}{1-z}\right)^{2s+1-m}
\left(\frac{z^3}{1-z}\right)^{m}\nonumber\\
&=&(1-z)^{-1}\left(\frac{z^{4}}{(1-z^2)(1-z)^2-(2z-z^2)
z^{4}}\right)^s \, (z^3)^m.
\end{eqnarray*}
Since for any $\gamma,\gamma_1\in \mathcal{I}_2(s,m)$
\index{$\mathcal{I}_k(s,m)$}
we have $\mathbf{Q}_{\gamma}(z)=\mathbf{Q}_{\gamma_1}(z)$, we derive
\begin{equation*}
{\bf Q}_{2}(z) = \sum_{m\geq 0}\sum_{\gamma\in\,{\mathcal
I}_2(m)} \mathbf{Q}_\gamma(z) =
\sum_{s\geq 0}\sum_{m=0}^s i_2(s,m)\mathbf{Q}_\gamma(z).
\end{equation*}
We set
\begin{equation*}\label{E:eta}
\eta(z)=\frac{z^{4}}{(1-z^2)(1-z)^2-(2z-z^2)z^{4}}
\end{equation*}
and note that Lemma~\ref{T:gfIk} guarantees
\begin{eqnarray*}
\sum_{s\geq0}\,\sum_{m=0}^{s}
\,i_2(s,m)\,x^s\,y^m & = & \frac{1+x}{1+2x-xy}\sum_{s\geq
0}f_2(2s,0)\left(\frac{x(1+x)}{(1+2x-xy)^2}\right)^s.
\end{eqnarray*}
Therefore, setting $x=\eta(z)$ and $y=z^3$ we arrive at
\begin{eqnarray*}
{\bf Q}_2(z)=\frac{1-z^2+z^4}{1 - z - z^2 + z^3 + 2 z^4 + z^6}\cdot {\bf F}_2
\left(\frac{z^4 - z^6 + z^8}{(1 - z - z^2 + z^3 + 2 z^4 + z^6)^2}\right)
\end{eqnarray*}
By Lemma~\ref{L:k-D}, ${\bf Q}_2(z)$ is $D$-finite. Pringsheim's
Theorem \cite{Tichmarsh:39} guarantees that ${\bf Q}_2(z)$ has a
dominant real positive singularity $\gamma_{2}$. We verify
that $\gamma_{2}$ which is the unique solution of minimum
modulus of the equation $\vartheta(z)=\rho_2^2$, where  $\rho_2^2$ is the unique
dominant singularity of ${\bf F}_2(z)$ and $\rho_2=1/2$. 
Furthermore we observe that
$\gamma_{2}$ is the unique dominant singularity of
${\bf Q}_2(z)$, see the SM.
It is straightforward to verify that
$\vartheta'(\gamma_{2})\neq 0$. According to
Proposition~\ref{P:algeasym}, we therefore have
\begin{equation*}
{\sf Q}_2(n)\sim c_2 n^{-3/2} \gamma_{2}^{-n},
\end{equation*}
and the proof of Proposition~\ref{P:k=2} is complete.
\end{proof}
%%%%%%%%%%%%%%%%%%%%%%%%%%%%%%%%%%%%%%%%%%%%%%%%%%%%%%%%%%%%%%%%%%%

%Via Remark~\ref{R:obacht} we discuss why the case $k=2$ is {\it not}
%implied by our main theorem. 

%%%
%%%%%%%%%%%%%%%%%%%%%%%%%%%%%%%%%%%%%%%%%%%%%%%%%%%%%%%%%%%%%%%%%%%%%%%%%%%%%%
%%%
\section{Colored shapes}\label{S:color}
%%%
%%%%%%%%%%%%%%%%%%%%%%%%%%%%%%%%%%%%%%%%%%%%%%%%%%%%%%%%%%%%%%%%%%%%%%%%%%%%%%
%%%

In the following we shall assume that $k>2$, unless stated otherwise.
The key to compute the generating function of modular $k$-noncrossing diagrams
are certain refinements of their ${\sf V}_k$-shapes. These refined shapes are 
called colored shapes and obtained by distinguishing a variety of crossings 
of $2$-arcs, i.e.~arcs of the form $(i,i+2)$. Each such class requires its 
specific inflation-procedure in Theorem~\ref{T:queen}.

Let us next have a closer look at these combinatorial classes (colors):
\begin{itemize}
\item $\mathbf{C}_1$\index{$\mathbf{C}_1$} the class of of $1$-arcs,
\item $\mathbf{C}_2$\index{$\mathbf{C}_2$}
      the class of arc-pairs consisting of mutually crossing
      $2$-arcs,
\item $\mathbf{C}_3$\index{$\mathbf{C}_3$} the class of arc-pairs
      $(\alpha,\beta)$ where $\alpha$ is the unique $2$-arc crossing
      $\beta$ and $\beta$ has length at least three.
\item $\mathbf{C}_4$\index{$\mathbf{C}_4$}
      the class of arc-triples $(\alpha_1,\beta,\alpha_2)$, where
      $\alpha_1$ and $\alpha_2$ are $2$-arcs that cross $\beta$.
\end{itemize}
In Fig.\ \ref{F:map2} we illustrate how these classes are induced by
modular $k$-noncrossing diagrams.
%%%
%%%%%%%%%%%%%%%%%%%%%%%%%%%%%%%%%%%%%%%%%%%%%%%%%%%%%%%%%%%%%%%%%%%%%%%%%%%%%%%
%%%
\begin{figure}[h!t!b!p]
\centering
\includegraphics{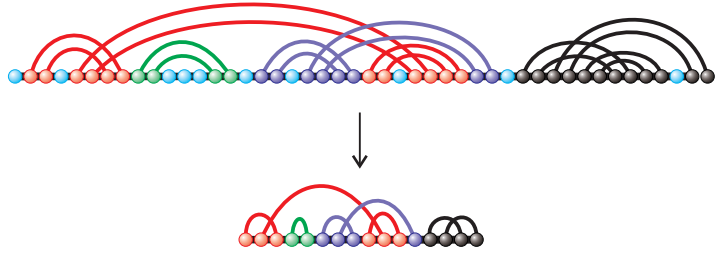}
\caption{\small Colored ${\sf V}_k$-shapes: a modular $3$-noncrossing diagram (top) 
and its colored ${\sf V}_3$-shape (bottom). In the resulting ${\sf
V}_3$-shape we color the four classes as follows:
$\mathbf{C}_1$(green), $\mathbf{C}_2$(black), $\mathbf{C}_3$(blue)
and $\mathbf{C}_4$(red).} \label{F:map2}
\end{figure}
%%%
%%%%%%%%%%%%%%%%%%%%%%%%%%%%%%%%%%%%%%%%%%%%%%%%%%%%%%%%%%%%%%%%%%%%%%%%%%%%%%%
%%%

Let us refine ${\sf V}_k$-shapes in two stages. For this
purpose let ${\mathcal I}_k(s,u_1,u_2)$\index{${\mathcal
I}_k(s,u_1,u_2)$} and $i_k(s,u_1,u_2)$\index{${\mathcal
I}_k(s,u_1,u_2)$} denote the set and cardinality of ${\sf V}_k$-shapes
having $s$ arcs, $u_1$ $1$-arcs and $u_2$ pairs of
mutually crossing $2$-arcs. Our first objective consists in
computing the generating function
\begin{equation*}
{\bf W}_k(x,y,w)=\sum_{s\geq 0}
\sum_{u_1=0}^{s}\sum_{u_2=0}^{\lfloor\frac{s-u_1}{2}\rfloor}
i_k(s,u_1,u_2)\, x^sy^{u_1}w^{u_2}.
\end{equation*}
That is, we first take the classes $\mathbf{C}_1$ and
$\mathbf{C}_2$ into account.

%%%
%%%%%%%%%%%%%%%%%%%%%%%%%%%%%%%%%%%%%%%%%%%%%%%%%%%%%%%%%%%%%%%%%%%%%%%%%%%%%%%
%%%
\begin{lemma}\label{L:recursion-0}
For $k> 2$, the coefficients $i_k(s,u_1,u_2)$ satisfy
\begin{eqnarray}\label{E:00}
i_k(s,u_1,u_2)& = & 0 \quad \text{for } u_1+2u_2>s\\
\label{E:wq}
\sum_{u_2=0}^{\lfloor\frac{s-u_1}{2}\rfloor}i_k(s,u_1,u_2) & = & i_k(s,u_1),
\end{eqnarray}
where $i_k(s,u_1)$ denotes the number of ${\sf V}_k$-shapes
having $s$ arcs, $u_1$ $1$-arcs. Furthermore we have the recursion:
\begin{eqnarray}\label{E:3rec}
(u_2+1)i_k(s+1,u_1,u_2+1)&=&(u_1+1)i_k(s,u_1+1,u_2)\nonumber\\&&
+(u_1+1)i_k(s-1,u_1+1,u_2).\label{E:2arcp}
\end{eqnarray}
\end{lemma}
%%%
%%%%%%%%%%%%%%%%%%%%%%%%%%%%%%%%%%%%%%%%%%%%%%%%%%%%%%%%%%%%%%%%%%%%%%%%%%%%%%%
%%%
The proof of Lemma~\ref{L:recursion-0} is given in SM. We next proceed by
computing ${\bf W}_k(x,y,w)$.
%%%
%%%%%%%%%%%%%%%%%%%%%%%%%%%%%%%%%%%%%%%%%%%%%%%%%%%%%%%%%%%%%%%%%%%%%%%%%%%%%%%
%%%
\begin{proposition}\label{P:trivariate}
For $k>2$, we have
\begin{equation}\label{E:3gf}
{\bf W}_k(x,y,w)=(1+x)v \, {\bf F}_k\left(x(1+x)v^2\right),
\end{equation}
where $v=\left((1-w)x^3+(1-w)x^2+(2-y)x+1\right)^{-1}$.
\end{proposition}
%%%
%%%%%%%%%%%%%%%%%%%%%%%%%%%%%%%%%%%%%%%%%%%%%%%%%%%%%%%%%%%%%%%%%%%%%%%%%%%%%%%
%%%
\begin{proof}
According to Lemma~\ref{T:gfIk}, we have
\begin{equation*}
{\bf I}_k(z,u)  =  \frac{1+z}{1+2z-zu} {\bf
F}_k\left(\frac{z(1+z)}{(1+2z-zu)^2}\right).
\end{equation*}
This generating function is connected to ${\bf W}_k(x,y,z)$ via
eq.~(\ref{E:wq}) of Lemma~\ref{L:recursion-0} as follows: setting
$w=1$, we have ${\bf W}_k(x,y,1)={\bf I}_k(x,y)$.
The recursion of eq.~(\ref{E:2arcp}) gives rise to the partial differential
equation
\begin{eqnarray}\label{E:wdiff}
\frac{\partial {\bf W}_k(x,y,w)}{\partial w}&=&x\frac{\partial {\bf
W}_k(x,y,w)}{\partial y}+x^2\frac{\partial {\bf
W}_k(x,y,w)}{\partial y}.
\end{eqnarray}
We next show
\begin{itemize}
\item the function
\begin{eqnarray}
{\bf W}_k^*(x,y,w)&=&\frac{(1+x)}{(1-w)x^3+(1-w)x^2+(2-y)x+1}
\times \nonumber\\
&&{\bf
F}_k\left(\frac{(1+x)x}{((1-w)x^3+(1-w)x^2+(2-y)x+1)^2}\right)\qquad
\end{eqnarray}
is a solution of eq.~(\ref{E:wdiff}),
\item its coefficients,
$i_k^*(s,u_1,u_2)=[x^sy^{u_1}w^{u_2}]{\bf
W}_k^*(x,y,w)$, satisfy $$i_k^*(s,u_1,u_2)=
0\quad\text{for}\quad u_1+2u_2>s,$$
\item ${\bf W}_k^*(x,y,1)={\bf I}_k(x,y)$.
\end{itemize}
Firstly,
\begin{eqnarray}\label{E:diffw}
\frac{\partial {\bf W}^*_k(x,y,w)}{\partial y} & = & u\,
{\bf F}_k\left(u\right)+ 2u \,{\bf F}_k'\left(u\right)\\
\frac{\partial {\bf W}^*_k(x,y,w)}{\partial w} & = & x(1+x)u \, {\bf
F}_k\left(u\right)+ 2x(1+x)u {\bf F}_k'\left(u\right),
\end{eqnarray}
where $$u= \frac{x(1+x)}{\left( (1-w)x^3+(1-w)x^2+(2-y)x+1\right)
^{2}}$$ and ${\bf F}_k'\left(u\right)=\sum_{n\geq
0}nf_k(2n)(u)^n$. Consequently, we derive
\begin{equation}
\frac{\partial {\bf W}_k^*(x,y,w)}{\partial w}=x\frac{\partial {\bf
W}_k^*(x,y,w)}{\partial y}+x^2\frac{\partial {\bf
W}_k^*(x,y,w)}{\partial y}.
\end{equation}
Secondly we prove $i_k^*(s,u_1,u_2)=0$ for $u_1+2u_2>s$.
To this end we observe that ${\bf W}^*_k(x,y,w)$ is a power series,
since it is analytic in $(0,0,0)$. It now suffices to note that the
indeterminants $y$ and $w$ only appear in form of products $xy$ and
$x^2w$ or $x^3w$.
Thirdly, the equality ${\bf W}_k^*(x,y,1)={\bf I}_k(x,y)$ is obvious.\\
{\it Claim.}
\begin{equation}\label{E:eequal2}
{\bf W}^*_k(x,y,w)={\bf W}_k(x,y,w).
\end{equation}
By construction the coefficients $i^*_k(s,u_1,u_2)$ satisfy
eq.~(\ref{E:3rec}) and we just proved $i_k^*(s,u_1,u_2)=0$ for $u_1+2u_2>s$.
In view of ${\bf W}_k^*(x,y,1)={\bf I}_k(x,y)$ we have
\begin{equation*}
\forall\, s,u_1;\qquad
 \sum_{u_2=0}^{\lfloor\frac{s-u_1}{2}\rfloor}i_k^*(s,u_1,u_2)=i_k(s,u_1).
\end{equation*}
Using these three properties it follows via induction over $s$
\begin{equation*}
\forall\, s,u_1,u_2\ge 0;\qquad i_k^*(s,u_1,u_2)=i_k(s,u_1,u_2),
\end{equation*}
whence the Claim and the proposition follows.
\end{proof}
In addition to $\mathbf{C}_1$ and $\mathbf{C}_2$, we consider next
the classes $\mathbf{C}_3$ and $\mathbf{C}_4$. For this purpose we
have to identify two new recursions, see Lemma~\ref{L:recursions-0}.
Setting $\vec{u}=(u_1,\dots,u_4)$ we denote by
$\mathcal{I}_k(s,\vec{u})$\index{$\mathcal{I}_k(s,\vec{u})$} and
$i_k(s,\vec{u})$\index{$i_k(s,\vec{u})$}, the set and number of
colored ${\sf V}_k$-shapes over $s$ arcs, containing $u_i$
elements of class $\mathbf{C}_i$, where $1\le i\le 4$. The key
result is
%%%
%%%%%%%%%%%%%%%%%%%%%%%%%%%%%%%%%%%%%%%%%%%%%%%%%%%%%%%%%%%%%%%%%%%%%%%%%%%%%%%
%%%
\begin{lemma}\label{L:recursions-0}
For $k>2$, the coefficients $i_k(s,\vec{u})$ satisfy
\begin{eqnarray}\label{E:erni2}
i_k(s,u_1,u_2,u_3,u_4) & = & 0  \quad\text{for }u_1+2u_2+2u_3+3u_4>s\\
\label{E:5ini}
\sum_{u_3,u_4\geq 0}i_k(s,u_1,u_2,u_3,u_4) & = & i_k(s,u_1,u_2).
\end{eqnarray}
Furthermore we have the recursions
\begin{align}
(u_3+1)&i_k(s+1,u_1,u_2,u_3+1,u_4)=\nonumber \\
&\quad \,  2u_1i_k(s-1,u_1,u_2,u_3,u_4)\nonumber\\
&+4(u_2+1)i_k(s-1,u_1,u_2+1,u_3,u_4)\nonumber\\
&+4(u_2+1)i_k(s-1,u_1,u_2+1,u_3-1,u_4)\nonumber\\
&+4(u_2+1)i_k(s-2,u_1,u_2+1,u_3-1,u_4)\nonumber\\
&+2(u_3+1)i_k(s,u_1,u_2,u_3+1,u_4)\nonumber\\
&+2u_3i_k(s-1,u_1,u_2,u_3,u_4)\nonumber\\
&+{6(u_3+1)i_k(s-1,u_1,u_2,u_3+1,u_4)}\nonumber\\
&+2(u_3+1)i_k(s-2,u_1,u_2,u_3+1,u_4)\nonumber\\
&+2u_3i_k(s-2,u_1,u_2,u_3,u_4)\nonumber\\
&+4(u_4+1)i_k(s,u_1,u_2,u_3-1,u_4+1)\nonumber\\
&+4(u_4+1)i_k(s-1,u_1,u_2,u_3-1,u_4+1)\nonumber\\
&+4u_4i_k(s-1,u_1,u_2,u_3,u_4)\nonumber\\
&+4(u_4+1)i_k(s-1,u_1,u_2,u_3,u_4+1)\nonumber\\
&+4u_4i_k(s-2,u_1,u_2,u_3,u_4)\nonumber\\
&+2(u_4+1)i_k(s-2,u_1,u_2,u_3,u_4+1)\nonumber\\
&+({2s}-2u_1-4u_2-4u_3-6u_4)i_k(s,u_1,u_2,u_3,u_4)\nonumber\\
&+2(2(s-1)-2u_1-4u_2-4u_3-6u_4)i_k(s-1,u_1,u_2,u_3,u_4)\nonumber\\
&+({2(s-2)}-4u_2-4u_3-6u_4)i_k(s-2,u_1,u_2,u_3,u_4)\label{u3recursion}
\end{align}
and
\begin{align}
2(u_4+1)i_k(s+1,u_1,u_2,u_3,u_4+1)&=(u_3+1)i_k(s,u_1,u_2,u_3+1,u_4)\nonumber\\
&\quad +2(u_2+1)_k(s,u_1,u_2+1,u_3,u_4).\quad\label{u4recursion}
\end{align}
\end{lemma}
%%%
%%%%%%%%%%%%%%%%%%%%%%%%%%%%%%%%%%%%%%%%%%%%%%%%%%%%%%%%%%%%%%%%%%%%%%%%%%%%%%%%%%
%%%
The proof of Lemma~\ref{L:recursions-0} is presented the SM and follows by
removing a specific arc in a labelled ${\bf C}_3$-element and careful accounting
of the resulting arc-configurations.

Proposition~\ref{P:trivariate} and Lemma~\ref{L:recursions-0} put us
in position to compute the generating function of colored ${\sf
V}_k$-shapes
\begin{equation}
{\bf I}_k(x,y,z,w,t)=
\sum_{s,u_1,u_2,u_3,u_4}i_k(s,\vec{u})\, x^s y^{u_1} z^{u_2}w^{u_3} t^{u_4}.
\end{equation}
%%%
%%%%%%%%%%%%%%%%%%%%%%%%%%%%%%%%%%%%%%%%%%%%%%%%%%%%%%%%%%%%%%%%%%%%%%%%%%%%%%%%%%
%%%
\begin{proposition}\label{P:5tri}
For $k>2$, the generating function of colored ${\sf V}_k$-shapes is given by
\begin{equation}
{\bf I}_k(x,y,z,w,t)=\frac{1+x}{\theta}{\bf F}_k
\left(\frac{x(1+(2w-1)x+(t-1)x^2)}{\theta^2}\right),
\end{equation}
where $\theta=1-(y-2)x+(2w-z-1)x^2+(2w-z-1)x^3$.
\end{proposition}
%%%
%%%%%%%%%%%%%%%%%%%%%%%%%%%%%%%%%%%%%%%%%%%%%%%%%%%%%%%%%%%%%%%%%%%%%%%%%%%%%%%%
%%%
\begin{proof}
The first recursion of Lemma~\ref{L:recursions-0} implies the partial
differential equation
\begin{align}
\frac{\partial {\bf I}_k}{\partial w}&=
\frac{\partial {\bf I}_k}{\partial x}(2x^2+4x^3+2x^4)-
\frac{\partial {\bf I}_k}{\partial y}(2xy+2x^2 y)\nonumber\\
&+\frac{\partial {\bf I}_k}{\partial z}
(-4xz+4x^2w+4x^2-4x^3z-8x^2z+4x^3w)\nonumber\\
&+\frac{\partial {\bf I}_k}{\partial w}(-4xw+2x-6x^2w+6x^2-2x^3w+2x^3)
\nonumber\\
&+\frac{\partial {\bf I}_k}{\partial t}(-6xt+4xw-8x^2t+4x^2w+4x^2-2x^3t+2x^3).
\label{E:KA}
\end{align}
Analogously, the second recursion of Lemma~\ref{L:recursions-0} gives rise to the
partial differential equation
\begin{align}
 2\frac{\partial {\bf I}_k}{\partial t}
&=\frac{\partial {\bf I}_k}{\partial w}x+
\frac{\partial {\bf I}_k}{\partial z}\, 2x.\label{E:KB}
\end{align}
Aside from being a solution of eq.~(\ref{E:KA}) and
eq.~(\ref{E:KB}), we take note of the fact that eq.~(\ref{E:5ini})
of Lemma~\ref{L:recursions-0} is equivalent to
\begin{equation}
{\bf I}_k(x,y,z,1,1)={\bf W}_k(x,y,z).\label{bound}
\end{equation}
We next show
\begin{itemize}
\item The function
\begin{align*}
{\bf I}_k^*(x,y,z,w,t)&=\frac{1+x}{1-(y-2)x+(2w-z-1)x^2+(2w-z-1)x^3}
\times \\
&{\bf F}_k
\left(\frac{x(1+(2w-1)x+(t-1)x^2)}
{(1-(y-2)x+(2w-z-1)x^2+(2w-z-1)x^3)^2}\right)
\end{align*}
is a solution of eq.~(\ref{E:KA}) and eq.~(\ref{E:KB}),
\item its coefficients,
$i_k^*(s,u_1,u_2,u_3,u_4)=[x^sy^{u_1}z^{u_2}w^{u_3}t^{u_4}]{\bf
I}_k^*(x,y,z,w,t)$, satisfy
$$i_k^*(s,u_1,u_2,u_3,u_4)=0\quad\text{for}\quad u_1+2u_2+2u_3+3u_4>s,$$
\item ${\bf I}_k^*(x,y,z,1,1)={\bf W}_k(x,y,z)$.
\end{itemize}
We verify by direct computation that ${\bf I}_k^*(x,y,z,w,t)$
satisfies eq.~(\ref{E:KA}) as well as eq.~(\ref{E:KB}).
Next we prove $i_k^*(s,u_1,u_2,u_3,u_4)=0$ for $u_1+2u_2+2u_3+3u_4>s$.
Since ${\bf I}_k^*(x,y,z,w,t)$ is analytic in $(0,0,0,0,0)$, it is
a power series. As the indeterminants $y$, $z$, $w$ and $t$
appear only in form of products $xy$, $x^2z$ or $x^3z$, $x^2w$ or
$x^3w$, and $x^3t$, respectively, the assertion follows.\\
{\it Claim.}
\begin{equation*}
{\bf I}_k^*(x,y,z,w,t)={\bf I}_k(x,y,z,w,t).
\end{equation*}
By construction,
$i_k^*(s,\vec{u})$ satisfies the recursions eq.~(\ref{u3recursion})
and eq.~(\ref{u4recursion}) as well as
$i_k^*(s,u_1,u_2,u_3,u_4)=0$ for $u_1+2u_2+2u_3+3u_4>s$.
Eq.~(\ref{bound}) implies
\begin{equation*}
\sum_{u_3,u_4\geq 0}i_k^*(s,u_1,u_2,u_3,u_4)=i_k(s,u_1,u_2).
\end{equation*}
Using these propertites we can show via induction over $s$
\begin{equation*}
\forall\, s,u_1,u_2,u_3,u_4\geq 0;\qquad
i_k^*(s,u_1,u_2,u_3,u_4)=i_k(s,u_1,u_2,u_3,u_4)
\end{equation*}
and the proposition is proved.
\end{proof}

%%%
%%%%%%%%%%%%%%%%%%%%%%%%%%%%%%%%%%%%%%%%%%%%%%%%%%%%%%%%%%%%%%%%%%%%%%%%%%
%%%

%%%
%%%%%%%%%%%%%%%%%%%%%%%%%%%%%%%%%%%%%%%%%%%%%%%%%%%%%%%%%%%%%%%%%%%%%%%%%%
%%%
\section{The main theorem}\label{S:main}
%%%
%%%%%%%%%%%%%%%%%%%%%%%%%%%%%%%%%%%%%%%%%%%%%%%%%%%%%%%%%%%%%%%%%%%%%%%%%%
%%%

We are now in position to compute ${\bf Q}_{k}(z)$. All technicalities aside, 
we already introduced the main the strategy in the proof of 
Proposition~\ref{P:k=2}:
as in the case $k=2$ we shall take care of all ``critical'' arcs by specific
inflations.

%%%
%%%%%%%%%%%%%%%%%%%%%%%%%%%%%%%%%%%%%%%%%%%%%%%%%%%%%%%%%%%%%%%%%%%%%%%%%%%%%%%
%%%
\begin{theorem}\label{T:queen}
Suppose $k>2$, then
\begin{eqnarray}\label{E:gut0}
{\bf Q}_{k}(z) &=&
\frac{1-z^2+z^4}{q(z)}
{\bf F}_k\left(\vartheta(z)\right),
\end{eqnarray}
where 
\begin{eqnarray}
q(z)& = & 1-z-z^2+z^3+2z^4+z^6-z^8+z^{10}-z^{12}
\nonumber \\
\label{sing_eq}
\vartheta(z) & = & 
\frac{z^4(1-z^2-z^4+2z^6-z^8)}{q(z)^2}.
\end{eqnarray}
Furthermore, for $3\leq k\leq 9$, ${\sf Q}_k(n)$ satisfies
\begin{equation}\label{E:gut1}
{\sf Q}_k(n)\sim  c_{k}\, n^{-((k-1)^2+(k-1)/2)}\,
\gamma_{k}^{-n} ,\quad \text{for some
$c_{k}>0$,}\qquad
\end{equation}
where $\gamma_{k}$ is the minimal, positive real solution
of $\vartheta(z)=\rho_k^2$, see Table~\ref{T:tab20}.
\end{theorem}
%%%
%%%%%%%%%%%%%%%%%%%%%%%%%%%%%%%%%%%%%%%%%%%%%%%%%%%%%%%%%%%%%%%%%%%%%%%%%%%%%%%
%%%
\begin{table}[!h]
\tabcolsep 0pt
\begin{center}
\def\temptablewidth{1\textwidth}
{\rule{\temptablewidth}{1pt}}
\begin{tabular*}{\temptablewidth}{@{\extracolsep{\fill}}llllllllll}
$k$   & $ 3 $ & $ 4 $ & $ 5 $ & $ 6 $ &  $7$ & $8$ & $9$\\
\hline
$\theta(n)$ &$n^{-5}$ & $n^{-\frac{21}{2}}$ &
$n^{-18}$ & $n^{-\frac{55}{2}}$ &
$n^{-39}$ & $n^{-\frac{105}{2}}$ & $n^{-68}$
\vspace{3pt}\\
\hline $\gamma_{k}^{-1}$& $2.5410$ & $3.0132$ & $3.3974$ &
$3.7319$ & $4.0327$ & $4.3087$ & $4.5654$\\
\hline
\end{tabular*}
\end{center}
\caption{\small Exponential growth rates $\gamma_{k}^{-1}$ and
subexponential factors $\theta(n)$, for modular, $k$-noncrossing diagrams
with minimum arc length four.} \vspace*{-12pt}
\label{T:tab20}
\end{table}
%%%
%%%%%%%%%%%%%%%%%%%%%%%%%%%%%%%%%%%%%%%%%%%%%%%%%%%%%%%%%%%%%%%%%%%%%%%%%%%%%%%
%%%
\begin{proof}
Let $\mathcal{Q}_k$ denote the set of modular, $k$-noncrossing diagrams 
and let $\mathcal{I}_k$ and $\mathcal{I}_k(s,\vec{u})$ denote
the set of all ${\sf V}_k$-shapes and those having $s$ arcs and 
$u_i$ elements belonging to class $\mathbf{C}_i$, where $1\le i\le 4$. 
Then we have the surjective map,
\begin{equation*}
\varphi_k:\ \mathcal{Q}_k\rightarrow \mathcal{I}_k,
\end{equation*}
%%%%%%%%%%%%%%%%%%%%%%%%%%%%%%%%%%%%%%%%%%%%%%%%%%%%%%%%%%%%%%%%
inducing the partition
$\mathcal{Q}_k=\dot{\cup}_\gamma\varphi_k^{-1}(\gamma)$, where
$\varphi_k^{-1}(\gamma)$ is the preimage set of shape $\gamma$ under
the map $\varphi_k$. This partition allows us to organize ${\bf
Q}_k(z)$ with respect to colored ${\sf V}_k$-shapes,
$\gamma$, as follows:
\begin{equation}\label{E:tk2}
{\bf Q}_k(z)= \sum_{s,\vec{u}}
\sum_{\gamma\in\mathcal{I}_k(s,\vec{u})} {\bf Q}_\gamma(z).
\end{equation}
We proceed by computing the generating function ${\bf Q}_\gamma(z)$
following the strategy of Proposition~\ref{P:k=2}, also using the
notation therein. The key point is that the inflation-procedures are
specific to the $\mathbf{C}_i$-classes. In the following we will
inflate all ``critical'' arcs, i.e.~arcs that require the insertion
of additional isolated vertices in order to satisfy the minimum
arc length condition. In the following we refer to a stem different
from a $2$-stack as a $\dagger$-stem\index{$\dagger$-stem}.
Accordingly, the combinatorial class of $\dagger$-stems is given by
$(\mathcal{M}-\mathcal{R}^2)$.

\begin{itemize}
\item {\bf $\mathbf{C}_1$-class:} here we insert isolated vertices, see
Fig. \ref{F:step1},
%%%
%%%%%%%%%%%%%%%%%%%%%%%%%%%%%%%%%%%%%%%%%%%%%%%%%%%%%%%%%%%%%%%%%%%%%%%%%%%%%%%
%%%
\begin{figure}[h!t!b!p]
\centering
\includegraphics[width=0.5\textwidth]{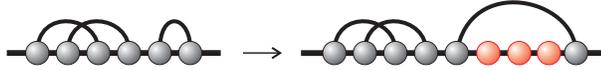}
\caption{\small $\mathbf{C}_1$-class: insertion of at least three vertices (red)
\label{F:step1}}
\end{figure}
%%%
%%%%%%%%%%%%%%%%%%%%%%%%%%%%%%%%%%%%%%%%%%%%%%%%%%%%%%%%%%%%%%%%%%%%%%%%%%%%%%%
%%%
and obtain immediately
\begin{equation}\label{E:K1}
{\bf C}_1(z)=\frac{z^3}{1-z}.
\end{equation}
%%%
%%%%%%%%%%%%%%%%%%%%%%%%%%%%%%%%%%%%%%%%%%%%%%%%%%%%%%%%%%%%%%%%%%%%%%%%%%%%%%%
%%%

\item {\bf $\mathbf{C}_2$-class:} any such element is a pair
$((i,i+2),(i+1,i+3))$ and we shall distinguish the following scenarios:
%%%
%%%%%%%%%%%%%%%%%%%%%%%%%%%%%%%%%%%%%%%%%%%%%%%%%%%%%%%%%%%%%%%%%%%%%%%%%%%%%%%
%%%
\begin{itemize}
\item both arcs are inflated to stacks of length two, see Fig. \ref{F:step2a}.
      Ruling out the cases where no isolated vertex is inserted and the
      two scenarios, where there is no insertion into the
      interval\index{interval}
      $[i+1,i+2]$ and only in either $[i,i+1]$ or $[i+2,i+3]$, see
      Fig.\ \ref{F:step2a}, we arrive at
\begin{equation*}
      \mathcal{C}_2^{(\text{a})}
      = \mathcal{R}^4 \times [(\textsc{Seq}(\mathcal{Z}))^3-\mathcal{E}
      -2(\mathcal{Z}\times\textsc{Seq}(\mathcal{Z}))].
\end{equation*}
      This combinatorial class has the generating function
\begin{equation*}
{\bf C}_{2}^{(\text{a})}(z)=
z^8\left(\left(\frac{1}{1-z}\right)^3-1-\frac{2z}{1-z}\right).
\end{equation*}
%%%
%%%%%%%%%%%%%%%%%%%%%%%%%%%%%%%%%%%%%%%%%%%%%%%%%%%%%%%%%%%%%%%%%%%%%%%%%%%%%%%
%%%
\begin{figure}[h!t!b!p]
\centering
\includegraphics[width=0.5\textwidth]{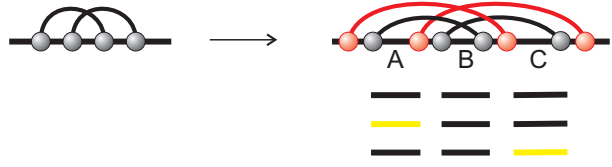}
\caption{\small {\bf $\mathbf{C}_2$-class:} inflation of both arcs to 2-stacks.
Inflated arcs are colored red while the original arcs of the
shape are colored black. We set $A=[i+1,i+2]$, $B=[i+2,i+3]$ and $C=[i+2,i+3]$
and illustrate the ``bad'' insertion scenarios as follows:
an insertion of some isolated vertices is represented by a yellow segment
and no insertion by a black segment. See the text for details.}
\label{F:step2a}
\end{figure}
%%%
%%%%%%%%%%%%%%%%%%%%%%%%%%%%%%%%%%%%%%%%%%%%%%%%%%%%%%%%%%%%%%%%%%%%%%%%%%%%%%%
%%%
\item one arc, $(i+1,i+3)$ or $(i,i+2)$ is inflated to a $2$-stack,
      while its counterpart is inflated to an arbitrary  $\dagger$-stem,
      see Fig. \ref{F:step2b}.
      Ruling out the cases where no vertex is inserted
      in $[i+1,i+2]$ and $[i+2,i+3]$ or $[i,i+1]$ and $[i+2,i+3]$, we obtain
\begin{equation*}
      \mathcal{C}_{2}^{(\text{b})}=2
      \mathcal{R}^2\times(\mathcal{M}-\mathcal{R}^2)
      \times((\textsc{Seq}(\mathcal{Z}))^2-\mathcal{E})
      \times\textsc{Seq}(\mathcal{Z}),
\end{equation*}
      having the generating function
    \begin{equation*}
      {\bf C}_{2}^{(\text{b})}(z)=2
      z^4\left(\frac{\frac{z^4}{1-z^2}}{1-\frac{z^4}{1-z^2}
      \left(\frac{2z}{1-z}+\left(\frac{z}{1-z}\right)^2\right)}-z^4\right)
      \left(\left(\frac{1}{1-z}\right)^2-1\right)
      {\frac{1}{1-z}}.
\end{equation*}
%%%
%%%%%%%%%%%%%%%%%%%%%%%%%%%%%%%%%%%%%%%%%%%%%%%%%%%%%%%%%%%%%%%%%%%%%%%%%%%%%%%
%%%
\begin{figure}[h!t!b!p]
\centering
\includegraphics[width=0.6\textwidth]{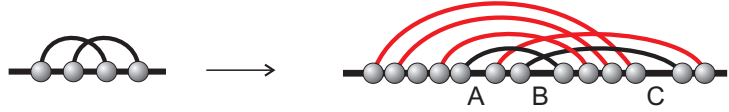}
\caption{\small 
{\bf $\mathbf{C}_2$-class:} inflation of only one arc to a $2$-stack.
Arc-coloring and labels as in Fig.\ \ref{F:step2a}}
\label{F:step2b}
\end{figure}
%%%
%%%%%%%%%%%%%%%%%%%%%%%%%%%%%%%%%%%%%%%%%%%%%%%%%%%%%%%%%%%%%%%%%%%%%%%%%%%%%%%
%%%
\item both arcs are inflated to an arbitrary $\dagger$-stem,
      respectively, see Fig.\ \ref{F:step2c}. In this case
      the insertion of isolated
      vertices is arbitrary, whence
\begin{equation*}
      \mathcal{C}_{2}^{(\text{c})}=
      (\mathcal{M}-\mathcal{R}^2)^2\times(\textsc{Seq}(\mathcal{Z}))^3,
\end{equation*}
      with generating function
    \begin{equation*}
      {\bf C}_{2}^{(\text{c})}(z)=
      \left(\frac{\frac{z^4}{1-z^2}}{1-\frac{z^4}{1-z^2}
      \left(\frac{2z}{1-z}+\left(\frac{z}{1-z}\right)^2\right)}-z^4\right)^2
      \left(\frac{1}{1-z}\right)^3.
\end{equation*}
\end{itemize}
%%%
%%%%%%%%%%%%%%%%%%%%%%%%%%%%%%%%%%%%%%%%%%%%%%%%%%%%%%%%%%%%%%%%%%%%%%%%%%%%%%%
%%%
\begin{figure}[h!t!b!p]
\centering
\includegraphics[width=0.6\textwidth]{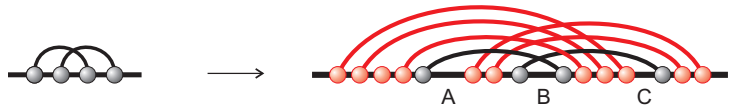}
\caption{\small
{\bf $\mathbf{C}_2$-class:} inflation of both arcs to an arbitrary
$\dagger$-stem. Arc-coloring and labels as in Fig.\ \ref{F:step2a}}
\label{F:step2c}
\end{figure}
%%%
%%%%%%%%%%%%%%%%%%%%%%%%%%%%%%%%%%%%%%%%%%%%%%%%%%%%%%%%%%%%%%%%%%%%%%%%%%%%%%%
%%%
As the above scenarios are mutually exclusive, the generating function of
the $\mathcal{C}_2$-class is given by
\begin{equation}\label{E:K2}
{\bf C}_2(z)= {\bf C}_2^{(\text{a})}(z)+{\bf
C}_2^{(\text{b})}(z)+{\bf C}_2^{(\text{c})}(z).
\end{equation}
Furthermore note that {\it both} arcs of the $\mathcal{C}_2$-class are inflated
in the cases (a), (b) and (c).

%%%
%%%%%%%%%%%%%%%%%%%%%%%%%%%%%%%%%%%%%%%%%%%%%%%%%%%%%%%%%%%%%%%%%%%%%%%%%%%%%%%
%%%
\item {\bf $\mathbf{C}_3$-class:} this class consists of arc-pairs $(\alpha,\beta)$
      where $\alpha$ is the unique $2$-arc crossing $\beta$ and $\beta$ has
      length at least three. Without loss of generality we can restrict our
      analysis to the case {$((i,i+2),(i+1,j))$}, $(j>i+3)$.
%%%
%%%%%%%%%%%%%%%%%%%%%%%%%%%%%%%%%%%%%%%%%%%%%%%%%%%%%%%%%%%%%%%%%%%%%%%%%%%%%%%
%%%
\begin{itemize}
\item the arc $(i+1,j)$ is inflated to a $2$-stack. Then
      we have to insert at least one isolated vertex in either
      $[i,i+1]$ or $[i+1,i+2]$, see Fig. \ref{F:step3}. Therefore we have
\begin{equation*}
    \mathcal{C}_{3}^{(\text{a})}=
    \mathcal{R}^2 \times (\textsc{Seq}(\mathcal{Z})^2-\mathcal{E}),
\end{equation*}
    with generating function
    \begin{equation*}
    {\bf C}_{3}^{(\text{a})}(z)= z^4\left(\left(\frac{1}{1-z}\right)^2-1\right).
    \end{equation*}
     Note that the arc $(i,i+2)$ is not considered here, it can be inflated
     without any restrictions.
\item the arc $(i+1,j)$ is inflated to an arbitrary  $\dagger$-stem,
      see Fig.\ \ref{F:step3}). Then
\begin{equation*}
    \mathcal{C}_{3}^{(\text{b})}=
    (\mathcal{M}-\mathcal{R}^2)\times
    {\textsc{Seq}(\mathcal{Z})^2,}
\end{equation*}
    with generating function
    \begin{equation*}
    {\bf C}_{3}^{(\text{b})}(z)=
    \left(\frac{\frac{z^4}{1-z^2}}{1-\frac{z^4}{1-z^2}
    \left(\frac{2z}{1-z}+\left(\frac{z}{1-z}\right)^2\right)}-z^4\right)
    \cdot \left(\frac{1}{1-z}\right)^2.
\end{equation*}
\end{itemize}
%%%
%%%%%%%%%%%%%%%%%%%%%%%%%%%%%%%%%%%%%%%%%%%%%%%%%%%%%%%%%%%%%%%%%%%%%%%%%%%%%%%
%%%
\begin{figure}[h!t!b!p]
\centering
\includegraphics[width=0.6\textwidth]{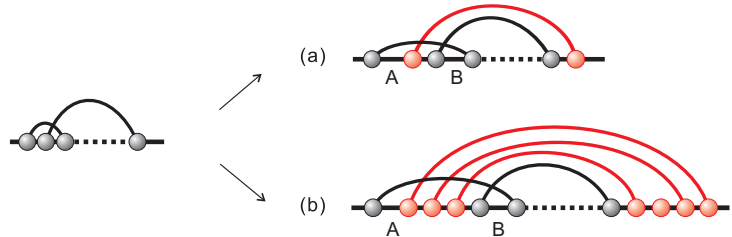}
\caption{\small $\mathbf{C}_3$-class: only one arc is inflated here and
its inflation distinguishes two subcases.
Arc-coloring as in Fig.\ \ref{F:step2a}}\label{F:step3}
\end{figure}
%%%
%%%%%%%%%%%%%%%%%%%%%%%%%%%%%%%%%%%%%%%%%%%%%%%%%%%%%%%%%%%%%%%%%%%%%%%%%%%%%%%
%%%
Consequently, this inflation process leads to a generating function
\begin{equation}\label{E:K3}
{\bf C}_3(z)={\bf C}_{3}^{(\text{a})}(z)+{\bf C}_{3}^{(\text{b})}(z).
\end{equation}
Note that during inflation (a) and (b) only {\it one} of the two arcs
of an $\mathbf{C}_3$-class element is being inflated.

\item {\bf $\mathbf{C}_4$-class:} this class consists of arc-triples
     $(\alpha_1,\beta,\alpha_2)$, where $\alpha_1$ and $\alpha_2$
     are $2$-arcs, respectively, that cross $\beta$.
\begin{itemize}
\item $\beta$ is inflated to a $2$-stack,
      see Fig.\ \ref{F:step4}. Using similar arguments as in the case of
      $\mathbf{C}_3$-class, we arrive at
\begin{equation*}
    \mathcal{C}_{4}^{(\text{a})}=
     \mathcal{R}^2 \times (\textsc{Seq}(\mathcal{Z})^2-\mathcal{E})
     \times(\textsc{Seq}(\mathcal{Z})^2-\mathcal{E}),
\end{equation*}
    with generating function
\begin{equation*}
    {\bf C}_{4}^{(\text{a})}(z)=
    z^4\left(\left(\frac{1}{1-z}\right)^2-1\right)^2.
\end{equation*}
\item the arc $\beta$ is inflated to an arbitary
     $\dagger$-stem\index{$\dagger$-stem},
     see Fig.\ \ref{F:step4},
\begin{equation*}
     \mathcal{C}_{4}^{(\text{b})}=
     (\mathcal{M}-\mathcal{R}^2)\times
     {\textsc{Seq}(\mathcal{Z})^4,}
\end{equation*}
    with generating function
\begin{equation*}
    {\bf C}_{4}^{(\text{b})}(z)=
    \left(\frac{\frac{z^4}{1-z^2}}{1-\frac{z^4}{1-z^2}
    \left(\frac{2z}{1-z}+\left(\frac{z}{1-z}\right)^2\right)}-z^4\right)
    \cdot \left(\frac{1}{1-z}\right)^4.
\end{equation*}
\end{itemize}
%%%
%%%%%%%%%%%%%%%%%%%%%%%%%%%%%%%%%%%%%%%%%%%%%%%%%%%%%%%%%%%%%%%%%%%%%%%%%%%%%%%
%%%
\begin{figure}[h!t!b!p]
\centering
\includegraphics[width=0.8\textwidth]{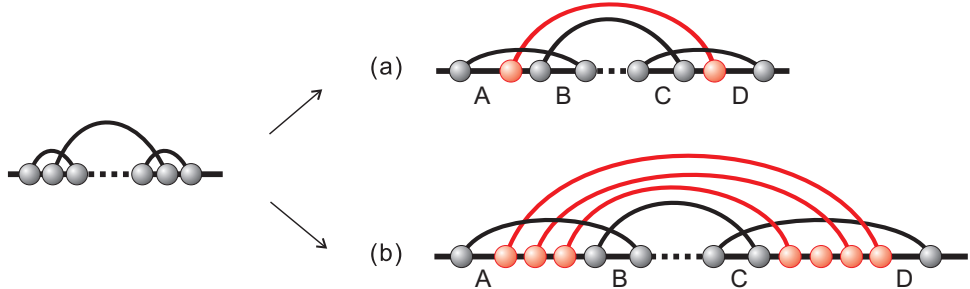}
\caption{\small $\mathbf{C}_4$-class: as for the inflation of $\mathbf{C}_3$
only the non $2$-arc is inflated, distinguishing two subcases.
Arc-coloring as in Fig.\ \ref{F:step2a}}
\label{F:step4}
\end{figure}
%%%
%%%%%%%%%%%%%%%%%%%%%%%%%%%%%%%%%%%%%%%%%%%%%%%%%%%%%%%%%%%%%%%%%%%%%%%%%%%%%%%
%%%
Accordingly we arrive at
\begin{equation}\label{E:K4}
{\bf C}_4(z)={\bf C}_{4}^{(\text{a})}(z)+{\bf C}_{4}^{(\text{b})}(z).
\end{equation}
\end{itemize}
The inflation of any arc of $\gamma$ not considered in the previous
steps follows the logic of Proposition~\ref{P:k=2}. We observe that
$(s-2u_2-u_3-u_4)$ arcs of the shape $\gamma$ have not been considered. 
Furthermore, $(2s+1-u_1-3u_2-2u_3-4u_4)$ intervals were not considered 
for the insertion of isolated vertices.
The inflation of these along the lines of Proposition~\ref{P:k=2} gives
rise to the class
\begin{equation*}
\mathcal{S}=
\mathcal{M}^{s-2u_2-u_3-u_4}\times
(\textsc{Seq}(\mathcal{Z}))^{2s+1-u_1-3u_2-2u_3-4u_4},
\end{equation*}
having the generating function
\begin{eqnarray*}
{\bf S}(z) & = & \left(\frac{\frac{z^4}{1-z^2}}{1-\frac{z^4}{1-z^2}
           \left(\frac{2z}{1-z}+\left(\frac{z}{1-z}\right)^2\right)}\right)
            ^{s-2u_2-u_3-u_4}\times \\
 &&  \qquad\qquad\qquad\qquad   \quad
      \left(\frac{1}{1-z}\right)^{2s+1-u_1-3u_2-2u_3-4u_4}.
\end{eqnarray*}
Since all these inflations can freely be combined, we have
\begin{eqnarray*}
\mathcal{Q}_\gamma = \mathcal{C}_1^{u_1}\times\mathcal{C}_2^{u_2}\times
                     \mathcal{C}_3^{u_3}\times\mathcal{C}_4^{u_4}\times
                      \mathcal{S},
\end{eqnarray*}
whence
\begin{align*}
{\bf Q}_\gamma(z) &=
{\bf C}_1(z)^{u_1}\cdot{\bf C}_2(z)^{u_2}\cdot
{\bf C}_3(z)^{u_3}\cdot{\bf C}_4(z)^{u_4}\cdot
{\bf S}(z)\\
&=\frac{1}{1-z}\varsigma_0(z)^s \varsigma_1(z)^{u_1} \varsigma_2(z)^{u_2}
\varsigma_3(z)^{u_3} \varsigma_4(z)^{u_4},
\end{align*}
where
\begin{align*}
\varsigma_0(z)&=\frac{z^4}{1 - 2 z + 2 z^3 - z^4 - 2 z^5 + z^6}
,\quad\varsigma_1(z)=z^3\\
\varsigma_2(z)&=
   \frac{z (1 - 4 z^3 + 2 z^4 + 8 z^5 - 6 z^6 - 7 z^7 + 8 z^8 + 2 z^9 -
   4 z^{10} + z^{11})}{1-z}\\
\varsigma_3(z)&=z (2 - 2 z^2 + z^3 + 2 z^4 - z^5)\\
\varsigma_4(z)&=z^2 (5 - 4 z - 3 z^2 + 6 z^3 + 2 z^4 - 4 z^5 + z^6).
\end{align*}
Observing that ${\bf Q}_{\gamma_1}(z)={\bf Q}_{\gamma_2}(z)$ for any
$\gamma_1,\gamma_2\in \mathcal{I}_k(s,\vec{u})$, we have according
to eq.~(\ref{E:tk2}):
\begin{equation*}
{\bf Q}_k(z)=\sum_{s,\vec{u}\ge 0}\, i_k(s,\vec{u})\;{\bf Q}_\gamma(z),
\end{equation*}
where $\vec{u}\ge 0$ denotes $u_i\geq0$ for $1\le i\le 4$.
Proposition \ref{P:5tri} guarantees
\begin{align*}
&\ \sum_{s,\vec{u}\geq0}i_k(s,\vec{u})\; x^n y^{u_1}
    z^{u_2} w^{u_3} t^{u_4}\\
&=\frac{1+x}{1-(y-2)x+(2w-z-1)x^2+(2w-z-1)x^3}\ \times \\
& \quad\  {\bf F}_k\left(\frac{x(1+(2w-1)x+(t-1)x^2)}
{(1-(y-2)x+(2w-z-1)x^2+(2w-z-1)x^3)^2}\right).
\end{align*}
Setting $x=\varsigma_0(z)$, $y=\varsigma_1(z)$, $z=\varsigma_2(z)$,
$w=\varsigma_3(z)$, $t=\varsigma_4(z)$, we arrive at
\begin{eqnarray*}
{\bf Q}_k(z)&=&
\frac{1-z^2+z^4}{1-z-z^2+z^3+2z^4+z^6-z^8+z^{10}-z^{12}}\ \times \\
& &{\bf F}_k\left(\frac{z^4(1-z^2-z^4+2z^6-z^8)}
{(1-z-z^2+z^3+2z^4+z^6-z^8+z^{10}-z^{12})^2}\right).
\end{eqnarray*}
By Lemma~\ref{L:k-D}, ${\bf Q}_k(z)$ is $D$-finite. Pringsheim's
Theorem \cite{Tichmarsh:39} guarantees that ${\bf Q}_k(z)$ has a
dominant real positive singularity $\gamma_{k}$. We verify
that for $3\leq k\leq 9$, $\gamma_{k}$ which is the unique
solution of minimum modulus of the equation $\vartheta(z)=\rho_k^2$
is the unique dominant singularity of ${\bf Q}_k(z)$,  and
$\vartheta'(z)\neq 0$, see the SM. According to
Proposition~\ref{P:algeasym} we therefore have
\begin{equation*}
{\sf Q}_k(n)\sim  c_{k}\, n^{-((k-1)^2+(k-1)/2)}\,
(\gamma_{k}^{-1})^n ,\quad \text{for some {$c_{k}^{}>0$}}\qquad
\end{equation*}
and the proof of Theorem~\ref{T:queen} is complete.
\end{proof}

\begin{remark}\label{R:obacht}{\rm
We remark that Theorem~\ref{T:queen} does not hold for $k=2$,
i.e.~we cannot compute the generating function ${\bf Q}_{2}(z)$ via
eq.~(\ref{E:gut0}). The reason is that Lemma~\ref{L:recursions-0}
only holds for $k>2$ and indeed we find
\begin{equation}\label{E:obacht}
{\bf Q}_2(z)\neq
\frac{1-z^2+z^4}{q(z)}
{\bf F}_2\left(\frac{z^4(1-z^2-z^4+2z^6-z^8)}{q(z)^2}\right).
\end{equation}
However, the computation of the generating function ${\bf Q}_2(z)$
in Proposition~\ref{P:k=2} is based on Lemma~\ref{T:gfIk}, which
does hold for $k=2$.}
\end{remark}
%%%
%%%%%%%%%%%%%%%%%%%%%%%%%%%%%%%%%%%%%%%%%%%%%%%%%%%%%%%%%%%%%%%%%%%%%%%%%%%%%%
%%%

%%%
%%%%%%%%%%%%%%%%%%%%%%%%%%%%%%%%%%%%%%%%%%%%%%%%%%%%%%%%%%%%%%%%%%%%%%%%%%
%%%
{\bf Acknowledgments.}
%%%
%%%%%%%%%%%%%%%%%%%%%%%%%%%%%%%%%%%%%%%%%%%%%%%%%%%%%%%%%%%%%%%%%%%%%%%%%%
%%%
This work was supported by the 973 Project, the PCSIRT Project of
the Ministry of Education, the Ministry of Science and Technology,
and the National Science Foundation of China.

\bibliographystyle{plain}

\begin{thebibliography}{00}


\bibitem{Chen}
Chen, W.Y.C., Deng, E.Y.P., Du, R.R.X., Stanley, R.P., Yan, C.H.
2007. {Crossings and nestings of matchings and partitions,} {\it
Trans. Amer. Math. Soc.} 359, 1555--1575.

\bibitem{Flajolet:07a}  Flajolet, P. and Sedgewick, R. 2009. {\it Analytic
combinatorics}, Cambridge University Press, New York.

\bibitem{Grabiner:93a}  Grabiner, D.J. and Magyar, P. 1993.
{Random walks in {Weyl} chambers and the decomposition of tensor
powers}, {\it J. Algebr. Comb.} 2, 239--260.

\bibitem{Reidys:08han}
 Han, H. S. W. and Reidys, C. M. 2008. Pseudoknot RNA structures with arc-lenght $\geq4$. {\it J. Comp. bio.}, 9, 1195--1208.

\bibitem{Stadler:99}
Haslinger, C. and Stadler, P.F., 1999. RNA structures with pseudo-knots. Bull. Math. Biol. 61, 437--467.

\bibitem{Schuster:98}
Hofacker, I.L., Schuster, P. and Stadler, P.F. 1998. {Combinatorics
of RNA secondary structures.}, {\it Discr. Appl. Math.} 88,
207--237.

\bibitem{Waterman:80}
 Howell, J.A., Smith, T.F.,  and Waterman, M.S. 1980. Computation  of generating functions for biological  molecules
{\it SIAM J. Appl. Math.} 39, 119¨C-133.

\bibitem{Reidys:07pseu} Jin, E.Y., Qin, J. and Reidys, C.M. 2008.
{Combinatorics of RNA structures with pseudoknots}, {\it Bull. Math.
Biol.} 70, 45--67.

\bibitem{Reidys:07asym}
Jin, E.Y. and Reidys, C.M. 2008. {Asymptotic enumeration of RNA
structures with pseudoknots}, {\it Bull. Math. Biol.} 70, 951--970.

\bibitem{Reidys:07lego}
Jin, E.Y. and Reidys, C.M. 2009. {Combinatorial design of pseudoknot
RNA}, {\it Adv. Appl. Math.} 42, 135--151.

\bibitem{Reidys:08k}
Jin, E.Y.,  Reidys, C.M. and Wang, R.R. 2008. { Asympotic analysis
of $k$-noncrossing matchings,} {\em arXiv:0803.0848}.

\bibitem{Reidys:08ma}
Ma, G. and Reidys, C.M. 2008. {Canonical RNA pseudoknot structures},
{\it J. Comput. Biol.} 15, 1257--1273.

\bibitem{Waterman:93}
 Penner, R.C. and Waterman, M.S. 1993. Spaces of RNA secondary {\it structures Adv. Math.} 101, 31--49.


\bibitem{Reidys:09shape}
Reidys, C.M. and Wang, R.R. 2009. Shapes of RNA pseudoknot structures, {\em arXiv:0906.3999}.

\bibitem{Rietveld:82}
 Rietveld, K., Van Poelgeest, R.,  Pleij, C.W., Van Boom, J.H., and Bosch, L. 1982. The tRNA-like structure at the $3'$ terminus of turnip yellow mosaic virus RNA. Differences and similarities with canonical tRNA. {\it Nucleic Acids Res}, 10,1929--1946.

\bibitem{Stanley:80}
Stanley, R. 1980. {Differentiably finite power series}, {\it Europ.
J. Combinatorics} 1, 175--188.


\bibitem{Tichmarsh:39}
Titchmarsh, E.C. 1939.{\it The theory of functions}, Oxford
Uninversity Press, Oxford, UK.

\bibitem{Tuerk:92}
Tuerk, C., MacDougal, S. and Gold, L. 1992. {RNA pseudoknots that
inhibit human immunodeficiency virus type 1 reverse transcriptase},
{\it Proc. Natl. Acad. Sci. USA} 89, 6988--6992.

\bibitem{Wasow:87}
Wasow, W. 1987. {\it Asymptotic expansions for ordinary differential
equations}, Dover, New York.

\bibitem{Waterman:78b}
Waterman, M.S. and Smith, T.F. 1978. RNA Secondary Structure: A  Complete Mathematical Analysis, {\it Mathematical Bioscience}, 42, 257--266.

\bibitem{Waterman:79}
Waterman, M.S. 1979. {Combinatorics of RNA hairpins and
cloverleafs}, {\it Stud. Appl. Math.} 60, 91--96.

\bibitem{Waterman:94a}
Waterman, M.S. and Schmitt, W.R. 1994. {Linear trees and RNA
secondary structure,} {\it Discr. Appl. Math.} 51, 317--323.

\end{thebibliography}

%%%%%%%%%%%%%%%%%%%%%%
%%%
%%%
%%%%%%%%%%%%%%%%%%%%%%%%%%%%%%%%%%%%%%%%%%%%%%%%%%%%%%%%%%%%%%%%%%%%%%
%%%

\end{document}